\documentclass[a4paper,11pt]{article}
\usepackage[english]{babel}
\usepackage[utf8]{inputenc}
\usepackage{graphicx}
\usepackage{amsmath}
\usepackage{amssymb}
\usepackage{amsthm}
\usepackage{amsfonts}
\usepackage{array}
\usepackage{amssymb}
\usepackage{mathrsfs}  
\usepackage{mathtools}
\usepackage{enumerate}
\usepackage{enumitem}
\usepackage{comment}
\usepackage{hyperref}
\usepackage[top=2cm, bottom=2cm, left=2cm, right=2cm]{geometry}
\usepackage{graphicx}
\usepackage{comment}
\usepackage{tikz}
\usepackage{xcolor}
\usepackage{pgfplots}
\usepackage{float}
\usepackage{eurosym}

\title{\bf
Tight Bound for Sum of Heterogeneous Random Variables: Application to Chance Constrained Programming
}
\date{}

\author{Quentin Jacquet$^{1,2}$, Riadh Zorgati$^{2}$
\\[5pt]
\small $^1$ INRIA, CMAP, Ecole Polytechnique, Palaiseau, France \\
\small $^2$ EDF R\&D Saclay, Palaiseau, France \\\small\texttt{\{quentin.jacquet,riadh.zorgati\}@edf.fr}
        }

\usepackage{enumerate}
\usepackage{dsfont,amsmath,amssymb,mathtools}

\DeclareMathOperator*{\Conv}{Conv}

\DeclareMathOperator{\diag}{diag}

\DeclareMathOperator*{\argmax}{arg\,max}

\DeclareMathOperator{\bbR}{\mathbb{R}}

\newcommand{\R}{\mathbb{R}}

\DeclareMathOperator{\Prob}{\mathbb{P}}

\newcommand{\shortminus}{\scalebox{0.5}[1.0]{$-$}}

\usepackage{hyperref}
\usepackage[capitalise]{cleveref}

\newtheorem{prop}{Proposition}[section]
\newtheorem{theorem}[prop]{Theorem}

\newtheorem{lemma}[prop]{Lemma}
\newtheorem{remark}{Remark}[section]

\DeclareMathSymbol{\mlq}{\mathord}{operators}{``}
\DeclareMathSymbol{\mrq}{\mathord}{operators}{`'}

\makeatletter
\newcommand*\frob{\mathpalette\bigcdot@{.7}}
\newcommand*\bigcdot@[2]{\mathbin{\vcenter{\hbox{\scalebox{#2}{$\m@th#1\bullet$}}}}}

\usepackage[colorinlistoftodos,bordercolor=orange,backgroundcolor=orange!20,linecolor=orange,textsize=scriptsize]{todonotes}

\usepackage[font=small,justification=centering]{caption}
\usepackage[justification=centering]{subcaption}

\usepackage{textcomp}

\usepackage{algorithm}
\usepackage[noend]{algpseudocode}

\algnewcommand{\IfThenElse}[3]{
  \State \algorithmicif\ #1\ \algorithmicthen\ #2\ \algorithmicelse\ #3}

\def\namedlabel#1#2{\begingroup
    #2%
    \def\@currentlabel{#2}%
    \phantomsection\label{#1}\endgroup
}

\usepackage[sorting=nty,backend=biber,url=false,isbn=false,eprint=true]{biblatex} 
\def\Prob{\mathbb{P}}

\def\Var{\textnormal{Var}}

\def\bbE{\mathbb{E}}
\def\bmean{\overline{b}}

\def\taumin{\tau^-}

\usetikzlibrary{shapes.multipart}
\usepackage{tikz-cd}

\usepackage{setspace}
\begin{document}
\maketitle

\begin{abstract}We study a tight Bennett-type concentration inequality for sums of heterogeneous and independent variables, defined as a one-dimensional minimization.
  We show that this refinement, which outperforms the standard known bounds, remains computationally tractable:  we develop a polynomial-time algorithm to compute confidence bounds, proved to terminate with an $\epsilon$-solution. 
From the proposed inequality, we deduce tight distributionally robust bounds to Chance-Constrained Programming problems. To illustrate the efficiency of our approach, we consider two use cases. First, we study the chance-constrained binary knapsack problem and highlight the efficiency of our cutting-plane approach by obtaining stronger solution than classical inequalities (such as Chebyshev-Cantelli or Hoeffding). Second, we deal with the Support Vector Machine problem, where the convex conservative approximation we obtain improves the robustness of the separation hyperplane, while staying computationally tractable.
\end{abstract}

~ 

\section{Introduction}

Concentration inequalities -- such as Hoeffding~\cite{Hoeffding_1963}, Bennett~\cite{Bennett_1962} or McDiarmid~\cite{McDiarmid_1989} to cite a few -- were originally introduced to quantify how a random variable deviates from their expectation. The crux of the matter is the imperfect knowledge of a random process: varying between the inequalities, the only available information are about the two first moments (mean and variance) or the length of the support of the distribution. These inequalities have now a wide variety of applications, see e.g.~\cite{Boucheron_2004}, including chance constrained programming or machine learning~\cite{Nemirovski_2007,Peng_2022,Wang_2015,Khanjani_2022}.

Many refinements of Hoeffding and Bennett's inequalities have been proposed: all these works exploit Chernoff's inequality but differ in the estimation of the moment-generating function $t\mapsto\bbE[e^{tX}]$. \Cref{fig::classification_mgf} proposes a schematic classification of the literature.
From and Swift~\cite{From_2013} and Zheng~\cite{Zheng_2017} both use a linear approximation of $x\mapsto e^{tx}$, that is tighter than Hoeffding' bound~\cite{Hoeffding_1963} for variables in $[0,1]$. They differ in the use of the arithmetic-geometric mean inequality.
Jebara~\cite{Jebara_2018} exploits an inequality
from~\cite[(b)]{Bennett_1962} to derive an analytic one-sided bound for sum of heterogeneous random variables.
Finally, Cheng and Li~\cite{Cheng_2022} insert a multipoint approximation of $e^{tX}$ and compare their results with~\cite{Zheng_2017}.
We emphazise that the classification we made -- which is a contribution on its own -- focuses on the crucial approximation done while tackling with Chernoff's inequalities, and does not directly compare the final bounds obtained in each work.

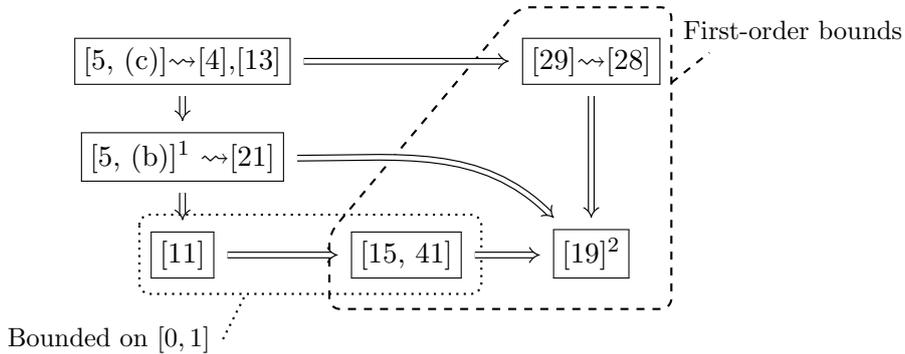
\begin{figure}[!ht]
\centering
\begin{tikzcd}[arrows=Rightarrow,
  column sep=small, row sep=small,
  /tikz/execute at begin picture={
    \draw[rounded corners, dotted,thick] (-3, -.75) rectangle (1.5, -1.8) {};
    \draw [rounded corners,dashed,thick] (-0.5,-2)--(-0.5,-1)--(2,2)--(4,2)--(4,-2)--cycle;
    \draw[black,thick,dashed](4,1) -- (4.5,1.4);
    \node (pprob2) at (5.6,1.7) { \small First-order bounds};
    \node (pprob2) at (-3.4,-2.4) { \small Bounded on $[0,1]$};
    \draw[black,thick,dotted](-1.9,-2.4) -- (-1.6,-1.8);
  }]
\fbox{\cite[(c)]{Bennett_1962}$\leadsto$\cite{Bennett_1968},\cite{Dembo_2010}}\arrow{d}{}\arrow{rr}{} & & \fbox{\cite{Pinter_1989}$\leadsto$\cite{Peng_2022}}\arrow{dd} \\
\fbox{\cite[(b)]{Bennett_1962}$^1$ $\leadsto$\cite{Jebara_2018}} \arrow{d}\arrow[out=0]{drr}{} &  \\
\fbox{\cite{Cheng_2022}} \arrow{r}{}& \fbox{\cite{From_2013,Zheng_2017}}\arrow{r}{} & \fbox{\cite{Hoeffding_1963}$^2$}
\end{tikzcd}
\caption{Classification of $t\mapsto \bbE[e^{tX}]$ estimations\\
$^{1}$ Bennett's inequality\\
$^{2}$ Hoeffding's inequality}
\label{fig::classification_mgf}
\medskip
\small If [a] $\Rightarrow$ [b], the upper estimator of the moment-generating function in [a] is tighter than in [b].
If [a] $\leadsto$ [b], [b] uses the same moment-generating estimator but improves / extends the results. Proofs of the different implications are provided in~\Cref{sec::comparison_moment_generating}.
\end{figure}

In this paper, we focus on the moment-generating estimator introduced in~\cite[(c)]{Bennett_1962}: for a random variable $X$ with mean $\mu$, variance $\sigma^2$ and such that $|X-\mu|\leq b$, we have the following inequality:
\begin{equation}\label{eq::mgf_bennett}
	\forall t\in\bbR,\;\bbE\left[e^{t X}\right] \leq e^{t\mu}\frac{\sigma^2 e^{|t| b} + b^2 e^{-\frac{|t|\sigma^2}{b }}}{b^2 + \sigma^2}\enspace.
\end{equation}This estimator is as tight as possible (knowing only $\mu$, $\sigma$ and $b$), since it has been proved to be exact for a particular Bernoulli distribution, see e.g.~\cite{Bennett_1962}. Dembo and Zeitouni~\cite{Dembo_2010} exploit this inequality but limit the study to identically distributed variables to obtain a closed-form expression involving a Kullback–Leibler divergence. Bennett \cite{Bennett_1968} extends the results to non identically distributed variables, but, in order to obtain explicit formula, further approximations have been made, leaving room for possible improvements. In contrast, we do not make additional approximations and directly construct the Chernoff bound using~\eqref{eq::mgf_bennett}. Even if an analytic solution is not known in the heterogeneous setting, we prove that this bound can be used in many applications.

We first focus on the computation of confidence bound and introduce a double bisection algorithm (\Cref{algo::refined_prob_bound}). We prove that this algorithm computes a bound with arbitrary precision in polynomial time (\Cref{prop::convergence_algo}).
This algorithm belongs to the class of Probabilistic Bisection Algorithms (PBA), see e.g.~\cite{Horstein_1963,Waeber_2013}, but instead of having a zero-mean noise, the error is bounded and controlled by a parameter.

We then apply this result on Chance-Constrained Programming (CCP)~\cite{Charnes_Cooper_1959,Miller_1965,Prekopa_1970,Henrion_2004,vanAckooij_2020}, a very attractive tool for dealing with uncertainty in optimization problems in addition to stochastic \cite{Birge_2011, Kall_1994, Ruszczynski_2003} and robust \cite{BenTal_2009, BenTal_2000} optimization approaches. This approach relies upon the characterization of uncertainty by means of probabilistic information and tries to find a good solution in a probabilistic sense. CCP aims at finding the best solution which satisfies uncertain constraints of the form $g(x, \xi) \geq 0$ with a given probability $p$, typically close to 1. A general CCP is expressed as:
\begin{eqnarray}\label{eq::CCP}
&\min& f[c(x,\xi)] \nonumber \\
&s.t.& \mathbb{P}[g_{i}(x, \xi) \geq 0, (i = 1, . . . ,m)] \geq p \label{eq:generalCCP} \\
&& x \in X \nonumber,
\end{eqnarray}
where $x \in \bbR^n$ denote a decision vector, $f$ is some general risk (for example $f(.)=\bbE{.}$) applied to the cost function $c$ that is impacted by uncertainty $\xi \in \bbR^m$ a general random vector/process, $\mathbb{P}$ is
the probability measure associated to the probability space
$(\Omega, \mathcal{F}, \mathbb{P})$ on which is defined $\xi$ applied to a whole system of $m$ stochastic inequalities $g_{i}(x, \xi) \geq 0, (i = 1, . . . ,m)$, $ p \in (0,1)$ is a given confidence level and $X$ models the deterministic feasible set for the decision vector $x$.

Although the underlying mathematical difficulties lead to very challenging tasks in the general case, CCP may lead to tractable algorithms. 
It is in particular the case for individual chance-constrained optimization 
$ \mathbb{P}[g_{i}(x, \xi) \geq 0] \geq p_{i}, \ (i = 1, . . . ,m)$ in Gaussian setting leading to a SOCP~\cite{Henrion_2007}, a convex optimization problem. 
More generally, for specific families of distributions, it is known that the set of a probabilistic constraint is convex, making possible the use of nonlinear methods, see e.g.~\cite{Prekopa_1995, Lagoa_2005,vanAckooij_Malick_2017}. 
However, the distributions are commonly unavailable in many applications, and even the evaluation of the constraint is not easy. 
A classical method is then to find a conservative approximation of the problem that is distributionally robust. In this case, the chance constraint are satisfied for any distribution and an optimal solution of the approximate problem gives a feasible solution of the original chance-constrained problem. 
Concentration inequalities have already been used in this context: to the best of our knowledge, Pinter~\cite{Pinter_1989} was the first to use concentration inequalities in optimization problem. Nemirovski and Shapiro~\cite{Nemirovski_2007} proved that the use of Chernoff's bounds provides tractable conservative approximation of chance-constrained problems. 
In particular, they detailed the convex approximation for several families of univariate distributions. 
Peng, Maggioni and Lisser~\cite{Peng_2022} focuses on SOCP conservative approximations of two types: distributionnaly robust formulations based on Hoeffding and Chebyshev's inequalities, and models that assumes  a normally distributed uncertainty. In particular, they deal with joint independent chance constraints, which is known to be of a high complexity, see e.g.~\cite{Nemirovski_2007}. 

Here, we compare various formulations on knapsack problems~\cite{Calafiore_2006,Ryu_2021} by specializing the study to second-order bounds (knowledge of means and variances). To this purpose,  we introduce a new convex conservative approximation based on Bernstein's inequality (\Cref{prop::convex_knapsack}) and derive from the tight Bennett's inequality a strong approximation (\Cref{prop::CKP-B}). We show that the two formulations can be efficiently solved by a cutting-plane approach and lead to solution improvement on instances from the literature (\Cref{table::knapsack_results}). In particular, for a given budget, we improve the objective value compared to the SOCP formulations~\cite{Peng_2022}.

Finally, we focus on the Support Vector Machine (SVM) problem with uncertainties, where the difficulty lies in the large number of probabilistic constraints. The distributionally robust version of the problem has been addressed in~\cite{Ben_Tal_2010,Wang_2015,Khanjani_2022}. 
In particular, Ben-Tal et al.~\cite{Ben_Tal_2010} first consider the same moment-generating estimator~\eqref{eq::mgf_bennett}, but make additional approximations in order to obtain SOCP formulations.  
Using convex optimization tools, we numerically highlight that our approach increases further the quality of the separation hyperplane while staying tractable for instances of substantial size.

This paper is organized as follows. In~\Cref{sec::theory}, we
first derive properties of the proposed inequality, and numerically observe its asymptotic behavior. Then, we introduce in~\Cref{sec::computing} an algorithm to compute confidence bounds. Finally, in~\Cref{sec::application}, we apply the inequality to chance-constrained programming, focusing on knapsack problems~(\Cref{sec::knapsack}) and on Support Vector Machine problem~(\Cref{sec::SVM}).

\paragraph{Notations.} For two vectors $a,b$ of $\bbR^N$, we denote by $\left<a,b\right>$ the Euclidean scalar product. Moreover, $a\,\wedge\,b$ (resp. $a\,\vee\,b$) stands for the component-wise maximum (resp. the minimum) between $a$ and $b$. Besides, for a discrete set $X$, $\Conv(X)$ will read as the convex envelope of $X$.



\section{On the tightest Cramér-Chernoff bound}\label{sec::theory}


We first recall Hoeffding's~\cite{Hoeffding_1963} and Bennett's~\cite{Bennett_1962} inequalities:
\begin{prop}[Hoeffding]\label{prop::Hoeffding}
Let $X_1,\hdots,X_N$ be $N$ independent random variables such that 
$\Prob[a_k \leq X_k - \bbE[X_k] \leq b_k] = 1$ for all $k\in\{1,\hdots,N\}$.
Then, for all $d\geq 0$,
\begin{equation}\label{eq::Hoeffding}
\ln\;\Prob\left[\sum_{k=1}^N X_k - \bbE[X_k] \geq d\right] \leq -\frac{2d^2 }{\sum_{k=1}^N (b_k - a_k)^2}\enspace.
\end{equation}
As a consequence, for all $\tau\in]0,1[$, $\Prob\left[\sum_{k=1}^N X_k - \bbE[X_k] \geq d_\tau\right] \leq \tau$ where
$
d_\tau =  \|b-a\|_2 \sqrt{-\ln\left(\sqrt{\tau}\right)} \enspace.
$
\end{prop}
\begin{prop}[Bennett]\label{prop::Bennett}
Let $X_1,\hdots,X_N$ be $N$ independent random variables such that 
\begin{enumerate}[label=(\roman*)]
\item $\Prob[X_k - \bbE[X_k] \leq b] = 1$, $k\in\{1,\hdots,N\}$,
\item $\sum_{k=1}^N\bbE[X^2_k] \leq \sigma^2$.
\end{enumerate}
Then, with $g:u\mapsto (1+u)\ln(1+u) - u$, we get for all $d\geq 0$,
\begin{equation}\label{eq::Bennett}
\ln\;\Prob\left[\sum_{k=1}^N X_k - \bbE[X_k] \geq d\right] \leq -\frac{\sigma^2}{b^2}g\left(\frac{bd}{\sigma^2}\right)\enspace.
\end{equation}
As a consequence, for all $\tau\in]0,1[$, $\Prob\left[\sum_{k=1}^N X_k - \bbE[X_k] \geq d_\tau\right] \leq \tau$ where
$
d_\tau = \frac{\sigma^2}{b}g^{\shortminus 1}\left(\frac{b^2}{\sigma^2}\ln\left(\frac{1}{\tau}\right)\right) \enspace.
$
\end{prop}

\Cref{prop::Hoeffding} and \Cref{prop::Bennett} does not suppose the same \emph{a priori} knowledge on the random variables: in the latter, information in second-moment is supposed whereas the former only needs knowledge on the mean of each random variable.
We now focus on the tightest second-order Cramér-Chernoff bound, firstly introduced in~\cite{Bennett_1962}, and based on~\eqref{eq::mgf_bennett}:
\begin{theorem}[Refined Bennett's inequality~\cite{Nemirovski_2007},Table 2]\label{prop::mgf}
Let $X_1,\hdots,X_N$ be $N$ independent random variables such that
\begin{enumerate}[label=(\roman*)]
\item $\Prob[X_k - \bbE[X_k]\leq b_k] = 1$, $k\in\{1,\hdots,N\}$,
\item $\Var(X_k) \leq \sigma^2_k$, $k\in\{1,\hdots,N\}$.
\end{enumerate}
Then, introducing $\gamma_k := \frac{\sigma_k^2}{b_k^2}$, for all $d\geq 0$
\begin{equation}\label{eq::pos_sum}
\forall \lambda\in\bbR_+^N,\quad\ln\;\Prob\left[\left<\lambda, X-\bbE[X]\right>\geq d\right] \leq \inf_{t> 0}\left\{-td 
+ \sum_{k=1}^N\ln\left(\frac{\gamma_k e^{t \lambda_k b_k}+e^{-t \lambda_k b_k \gamma_k}}{1+\gamma_k}\right)\right\}\enspace.
\end{equation}
In addition, if $\Prob[X_k - \bbE[X_k]\geq -b_k] = 1$, 
\begin{equation}\label{eq::qq_sum}
\forall \lambda\in\bbR^N,\quad\ln\;\Prob\left[\left<\lambda,X-\bbE[X]\right>\geq d \right] \leq\inf_{t> 0}\left\{-td 
+ \sum_{k=1}^N\ln\left(\frac{\gamma_k e^{t |\lambda_k| b_k}+e^{-t |\lambda_k| b_k \gamma_k}}{1+\gamma_k}\right)\right\}\enspace.
\end{equation}
\end{theorem}
\begin{proof}
Using the Chernoff' inequality on the variable $\left<\lambda,X-\bbE[X]\right>$, we obtain
$$
\Prob\left[\left<\lambda,X-\bbE[X]\right>\geq d\right]\leq e^{- t (d+\left<\lambda,\bbE[X]\right>)}\;\bbE\left[e^{t \left<\lambda,X\right>}\right]\enspace.
$$
By the independence of the variables $X_k$, we have
$\bbE\left[e^{t \left<\lambda,X\right>}\right] = \prod_{k=1}^N
\bbE\left[e^{t\lambda_k X_k}\right]$.
Finally, using~\eqref{eq::mgf_bennett},
we obtain for all $t\geq 0$:
$$\Prob[\left<\lambda,X-\bbE[X]\right>\geq d] \leq e^{-td}
\prod_{k=1}^N \left(\frac{\gamma_k e^{t |\lambda_k|b_k} + e^{-t|\lambda_k|\gamma_k b_k}}{ 1+ \gamma_k}\right)\enspace.$$
We conclude by applying the logarithm and by rearranging the terms.
\end{proof}
The right-hand sides of~\eqref{eq::pos_sum} and~\eqref{eq::qq_sum} correspond to the Cramér transform~\cite[Section 2.2]{Dembo_2010} of the Bernoulli distribution that achieves the equality in~\eqref{eq::mgf_bennett}.
The scope of \Cref{prop::mgf} is slightly more general than Hoeffding and Bennett inequality since we allow to have sum of weighted heterogeneous random variables (positive or negative weights). 

Under the assumptions of~\Cref{prop::mgf}, and introducing $\taumin:=\prod_{k=1}^N\frac{\gamma_k}{1+\gamma_k}$, we get as an immediate corollary :
\begin{equation}\label{eq::varphi_star}
\ln\;\Prob\left[\sum_{k=1}^N X_k - \bbE[X_k] \geq \alpha N\right] \leq \varphi_\alpha^* :=\inf_{t\geq 0} \varphi_\alpha(t)\enspace,
\end{equation}
where $\overline{b} = \frac{1}{N}\sum_{k=1}^N b_k$ and 
\begin{equation}\label{eq::varphi_2}
\varphi_\alpha : t\geq0\mapsto \ln(\taumin) + N t\left(\bmean - \alpha\right) 
+ \sum_{k=1}^N\ln\left(1+\gamma_k^{\shortminus 1}e^{-t b_k (1+\gamma_k)}\right)\enspace.
\end{equation}
The expression of $\varphi_\alpha$ is derived from~\eqref{eq::pos_sum} with $\lambda_k=1$. 
In the specific case where the coefficient $\bbE[X_k],\sigma_k$ and $b_k$ are identical for all $k\in \{1\hdots N\}$, the minimization in $t$ that appears in~\eqref{eq::pos_sum} has an analytic solution (using Kullback-Leibler divergence), see e.g.~\cite{Dembo_2010,Raginsky_2013}. In the framework of this paper, we allow heterogeneous parameters, and therefore the minimum is no longer analytically known.
 Nonetheless, the following properties show that the one-dimensional minimization is well defined:
 
\begin{prop}[Study of $\varphi_\alpha$]\label{corol::prop_on_varphi}Let $\alpha\geq 0$, then $\varphi_\alpha(0) = 0$. Moreover, the mapping $\varphi_\alpha$ is twice differentiable and their respective derivatives are
\begin{enumerate}[label=(\roman*)]
\item $\displaystyle\frac{d}{dt}\varphi_\alpha(t) = N\left(\bmean - \alpha\right) - \sum_{k=1}^N \frac{b_k(1+\gamma_k)} {1+\gamma_k e^{t b_k (1+\gamma_k)}}\enspace,$
\item$\displaystyle\frac{d^2}{dt^2}\varphi_\alpha(t) = \sum_{k=1}^N b_k^2 (1+\gamma_k)^2\frac{\gamma_k e^{t b_k(1+\gamma_k)}}{\left(1+\gamma_k e^{t b_k(1+\gamma_k)}\right)^2}\enspace.$\\
Moreover, $0\leq \frac{d^2}{dt^2}\,\varphi_\alpha(t) \leq M:=\frac{1}{2}\sum_{k=1}^N b^2_k (1+\gamma_k)^2$.
\end{enumerate}
\end{prop}
\begin{proof}
Let us introduce for all $\gamma,d \in \R^*_+$,
$f(t) = \ln\left(1+\gamma^{\shortminus 1} e^{-td} \right)$.
It follows that
$
f'(t) = \frac{-d}{1+\gamma e^{td}}
$
and
$
f''(t) = d^2\frac{\gamma e^{td}}{\left(1+\gamma e^{td}\right)^2}\in \left[0,\frac{d^2}{2}\right]\enspace.$
To recover the result, note that $\varphi_\alpha(t)$ is the sum of functions $f(\cdot)$ with
$d_k = b_k(1+\gamma_k)$.
\end{proof}
We immediately deduce from~\Cref{corol::prop_on_varphi} that the function $\varphi_\alpha(\cdot)$ is strictly convex, and thus the position of the minimum, denoted by $t_\alpha^* \in \R_+\cup\{+\infty\}$, is unique. The following lemma lists useful properties of $ \varphi^*_\alpha$ that will be used in the sequel.

\begin{lemma}[Study of $\varphi^*_\alpha$]\label{lemma::prop_on_varphi_star}~
\begin{enumerate}[label=(\roman*)]
  \item For $\alpha\leq \bmean :=\frac{1}{N}\sum_{k=1}^N b_k$, the function $\alpha\mapsto\varphi^*_\alpha$ is decreasing. \\Moreover, $\varphi^*_0 = 0$, $\varphi^*_{\bmean} = \ln(\taumin)$ and for $\alpha > \bmean$, $\varphi^*_\alpha = -\infty$.
  \item For $\alpha < \min_k\{b_k\}$, $t^*_\alpha \in \Conv\left(\left\{t^{(k)}_\alpha\right\}_k\right)$, 
  where 
  $t^{(k)}_\alpha = \frac{1}{b_k(1+\gamma_k)} \ln\left(\frac{\alpha + b_k \gamma_k}{\gamma_k (b_k - \alpha)}\right)\enspace.$\\
  Moreover, $t^*_\alpha \leq -\frac{1}{N(\bmean-\alpha)}\ln(\tau^-)\enspace.$
  \item For $\alpha_1,\alpha_2 < \bmean$,\quad
  $\left|\varphi^*_{\alpha_2} - \varphi^*_{\alpha_1}\right| \geq N \min\{t^*_{\alpha_1},t^*_{\alpha_2}\} \left|\alpha_2 - \alpha_1\right|$.
\end{enumerate}
\end{lemma}
\begin{proof}
\begin{enumerate}[label=(\roman*)]
\item Let $\alpha < \beta$. 
There exits $t^*_\alpha$ such that $\varphi^*_\alpha = \varphi_\alpha(t^*_\alpha)$.
Besides, $\varphi_\alpha(t^*_\alpha) = Nt^*_\alpha(\beta - \alpha) + \varphi_{\beta}(t^*_\alpha) > \varphi_{\beta}(t^*_\alpha)$.
As $\varphi^*_{\beta} \leq \varphi_{\beta}(t^*_\alpha)$ by optimality, we easily conclude.\\
Then, if $\alpha = \bmean$, the infimum is reached for $t\to\infty$ and is equal to $\ln(\taumin)$. If now $\alpha = 0$, then the minimum is attained at $t=0$ ($\frac{d\varphi_\alpha}{dt}(t)\geq 0$).
Finally, if $\alpha > \bmean$, then all the terms that appear in $\varphi^*_\alpha$ are decreasing, and the function diverges to $-\infty$.

\item $t^{(k)}_\alpha$ would the minimum if there were only the $k$-term in $\varphi$. As the minimum of a sum of convex functions lies in the convex envelope of the set of minimizers of each term, we get the property.\\ 
Besides, by (i), $\varphi_\alpha(t^*_\alpha) \leq \varphi_\alpha(0) = 0$.
Therefore, $\ln(\tau^-) + Nt_\alpha^*(\bmean - \alpha) \leq 0$, and so $N t_\alpha^*\leq \frac{\shortminus 1}{\bmean - \alpha}\ln(\tau^-)$.

\item Let $\alpha_1,\alpha_2 \leq \bmean$. Then, by definition,
$$
\begin{aligned}
&\varphi^*_{\alpha_1} = \varphi_{\alpha_1}(t^*_{\alpha_1}) = N t^*_{\alpha_1} (\alpha_2 - \alpha_1) + \varphi_{\alpha_2}(t^*_{\alpha_1}) \\
&\varphi^*_{\alpha_2} = \varphi_{\alpha_2}(t^*_{\alpha_2}) = N t^*_{\alpha_2} (\alpha_1 - \alpha_2) + \varphi_{\alpha_1}(t^*_{\alpha_2}) 
\end{aligned}
$$
from which we deduce by optimality of $t^*_{\alpha_1}$ and $t^*_{\alpha_2}$:
$$
\begin{aligned}
&\varphi^*_{\alpha_1} \geq N t^*_{\alpha_1} (\alpha_2 - \alpha_1) + \varphi^*_{\alpha_2}\\
&\varphi^*_{\alpha_2} \geq N t^*_{\alpha_2} (\alpha_1 - \alpha_2) + \varphi^*_{\alpha_1}
\end{aligned}
$$
As a consequence, $\left|\varphi^*_{\alpha_2} - \varphi^*_{\alpha_1}\right|\geq N\min\{t^*_{\alpha_1},t^*_{\alpha_2}\}\left|\alpha_2 - \alpha_1\right|\enspace.$
\end{enumerate}
\end{proof}

The next theorem can be directly derived from~\Cref{lemma::prop_on_varphi_star} and  provides an alternative confidence bound $\alpha_\tau N$ to the bound $d_\tau$ provided in~\Cref{prop::Hoeffding} and~\Cref{prop::Bennett}.
\begin{theorem}\label{prop::main_result}
For all $\tau\in[\taumin,1[$, there exists a unique $\alpha_\tau$ such that $\varphi^*_{\alpha_\tau} = \ln\left(\tau\right)$. 
As a consequence,  $\Prob\left[\sum_{k=1}^N X_k - \bbE[X_k] \geq \alpha_\tau N\right] \leq \tau$.
\end{theorem}

\paragraph{Numerical experiments.}
We aim to numerically compare the bounds developed in~\Cref{sec::theory} with four inequalities: Hoeffding~\eqref{eq::Hoeffding}, Bennett~\eqref{eq::Bennett}, Cantelli (a one-sided improvement of Chebyshev's inequality, see e.g.~\cite{Boucheron_2004}) and the bound introduced by Jebara~\cite{Jebara_2018}.
To this purpose, 
we follow the methodology of~\cite{Jebara_2018}: we search to bound $\ln\;\Prob\left[\sum_{k=1}^N X_k - \bbE[X_k] \geq \alpha N\right]\enspace,$
where the parameters $\bbE[X_k],\sigma_k, a_k$, $b_k$ and $\alpha$ are randomly generated following the rules described in~\Cref{tab::random_var_comp}.
\begin{table}[!ht]
\centering
\begin{tabular}{r|l}
	\hline
  $\bbE[X_k]$ & $\mathcal{U}(0,1)$\\
  $a_k$ & $\mathcal{U}(-1,0)$\\
  $b_k$ & $\mathcal{U}(0,1)$\\
  $\sigma_k$ & $\mathcal{U}\left(0,(b_k - a_k)\,/\,2\right)$\\
  $\alpha$ & $\mathcal{U}\left(0,\bmean\right)$\\
  \hline
\end{tabular}
\caption{Definition of the random variables}\label{tab::random_var_comp}
\end{table}

In order to have a fast implementation of $\varphi^*_\alpha$, we introduce a bisection algorithm, see~\Cref{algo::refined_prob_bound}. Note that this bisection method is only valid because we have shown that $\varphi_\alpha$ is convex and $t^*_\alpha$ is bounded, see~\Cref{lemma::prop_on_varphi_star}. The four other bounds are immediate to compute as they are analytically known.
\begin{algorithm}[!ht]
\caption{Bisection Search to compute $\varphi^*_\alpha$}\label{algo::refined_prob_bound}
\begin{algorithmic}
\Require $N$, $\alpha$, $b_k$, $\sigma_k$, $\epsilon_t$
\State $t^-,t^+\gets 0,\shortminus\frac{1}{N(\bmean - \alpha)}\ln(\tau^-)$ 
  \While{$t^+ - t^- > \epsilon_t$ }
     \State $\hat{t} \gets \frac{1}{2}(t^- + t^+)$
     \State $g \gets \frac{d}{dt} \varphi_\alpha(\hat{t})$
      \IfThenElse{$g\geq 0$}
        {$t^+ \gets \hat{t}$}
        {$t^- \gets \hat{t}$}
  \EndWhile
~\\\Return $\hat{\varphi}$
\end{algorithmic}
\end{algorithm}

The result are depicted in~\Cref{fig::res} for 500 realizations and are performed on a laptop \texttt{Intel Core i7 @2.20GHz\,$\times$\,12}. 
The log-probability of the four methods are represented on the $y$-axis for each value of $\varphi_\alpha^*$, represented on the $x$-axis. 
We recover the results proved in~\Cref{sec::comparison_moment_generating}: 
$\varphi_\alpha^*$ always outperforms Bennett, Hoeffding and Jebara's inequalities. We observe that Cantelli's bound is better for large probability error (small $\alpha$) -- typically $\exp\;\varphi^*_\alpha\geq 20\%$ -- but becomes rapidly dominated by the four other Chernoff's inequalities. In fact, Chebyshev's inequality has a quadratic decay in $\alpha$ when the Chernoff's bounds has exponential behaviors . For $\varphi^*_\alpha \geq -5$, the bound from~\cite{Jebara_2018} may be less efficient. Possibly, this bound can exceed 1, because the minimizer that is used in the Chernoff's inequality has no guarantee to be optimal.

The computation of the Bennett's and Hoeffding's bounds is almost immediate. The refined version of~\cite{Jebara_2018} takes around $1$ms per instance for $N=100$ (due to the computation of Lambert function), and $\varphi_\alpha^*$ takes around $5$ms per instance for $N=100$ for precision $\epsilon_t = 1$e$\shortminus 6$.

\begin{remark}
We did not display the results for Bernstein's bound, as it is known that this inequality is strictly looser than Bennett's inequality~\cite{Bennett_1962}, see e.g.~\cite{Jebara_2018} for a proof.
\end{remark}

\begin{figure*}[!ht]
\centering
\begin{subfigure}{.99\textwidth}
  	\centering
	\includegraphics[width = 0.85\linewidth, clip=true,trim = 0.4cm 0.cm 1.8cm 1.2cm] {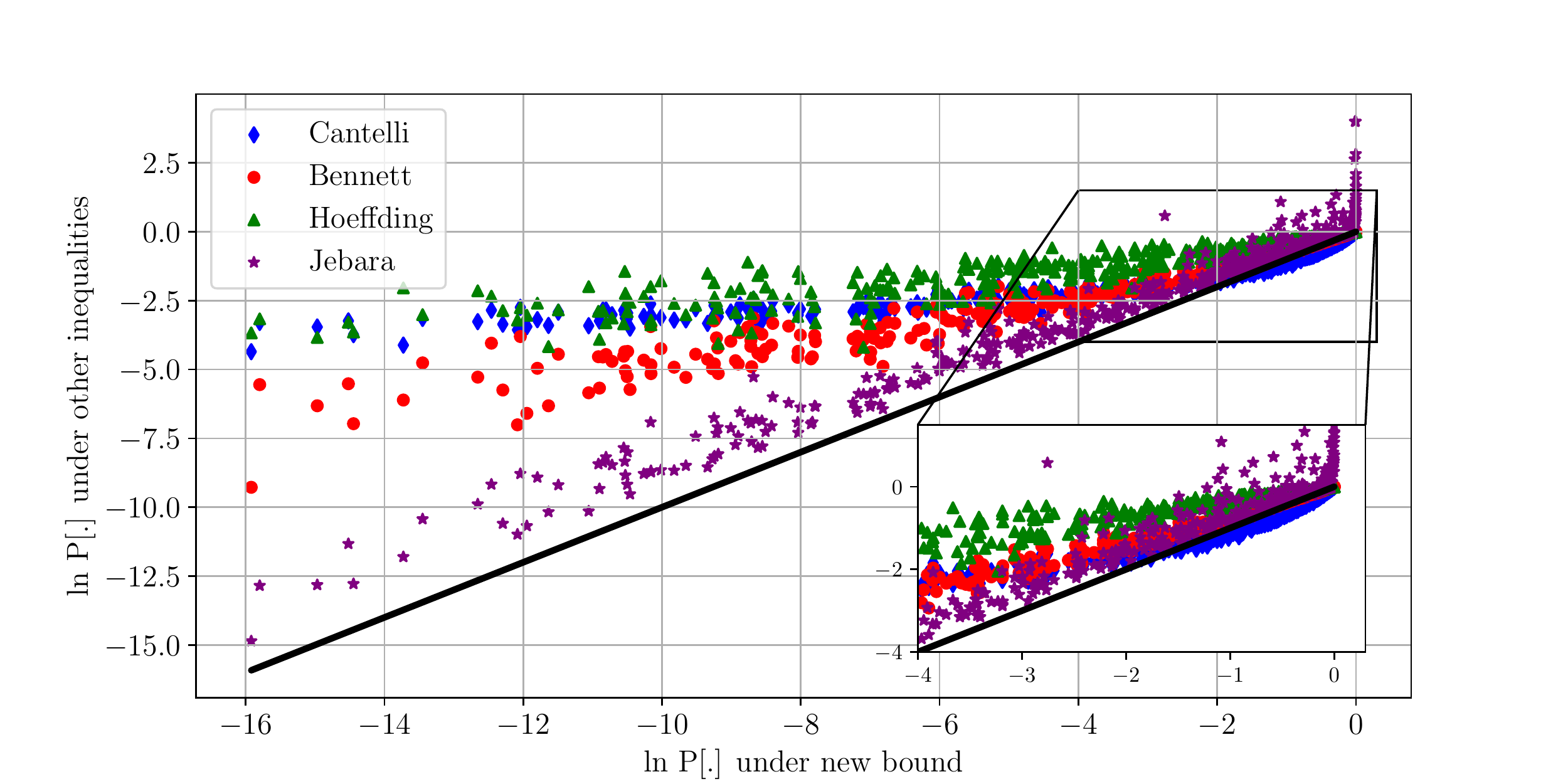}
	\caption{Random instances, made of $N=10$ heterogeneous variables}
	\label{fig::res_10}
\end{subfigure}\\
\begin{subfigure}{.99\textwidth}
  	\centering
	\includegraphics[width = 0.85\linewidth, clip=true,trim = 0.4cm 0.cm 1.8cm 1.2cm] {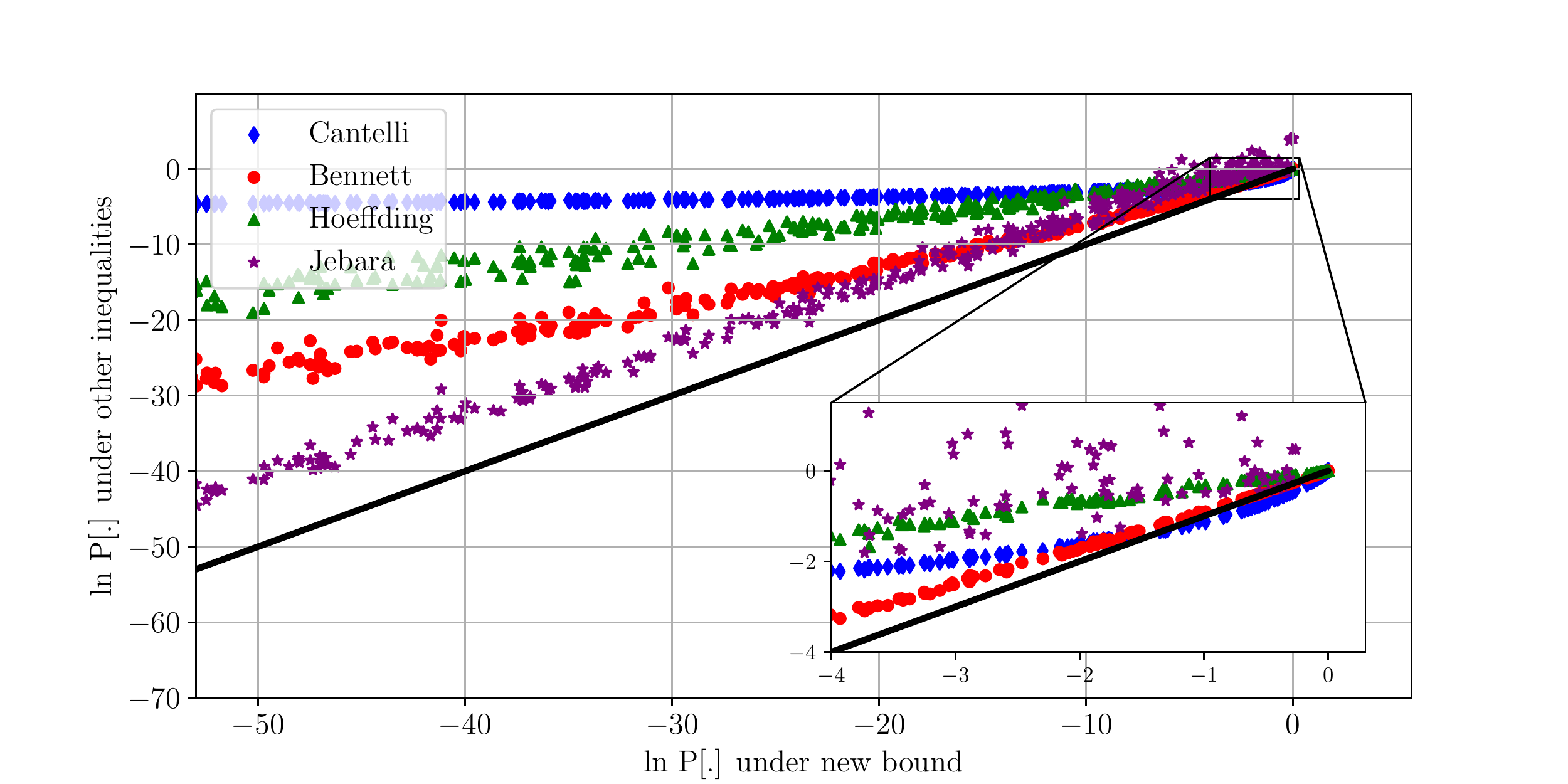}
	\caption{Random instances, made of $N=100$ heterogeneous variables}
	\label{fig::res_50}
\end{subfigure}
\caption{Comparison of four bounds with $\varphi^*_\alpha$.\\
 If the marker is 
above the line, then the method gives looser bound compared to $\varphi_\alpha^*$.}
\label{fig::res}
\end{figure*}

\section{Computing confidence bounds}\label{sec::computing}

In this section, we aim to derive a confidence bound $\alpha_\tau$ with a given maximum probability error $\tau$ such that $\varphi^*_{\alpha_\tau} = \ln(\tau)$.
We first prove that the confidence bound we obtain always provides useful information, as $\alpha_\tau$ is strictly lower than $\bmean$ for $\tau\geq\taumin$ (for probability error less than $\taumin$, we can only certify a confidence bound of $\bmean$) :
\begin{lemma}\label{prop::alpha_lower_b}
Suppose that $\tau \in ]\taumin,1]$, then 
$$\alpha_\tau \leq \bmean - \sqrt{\frac{1}{N\Gamma}\ln(\tau/\tau^-)}\enspace,$$
where $\Gamma = 1+(\min\{\gamma_k\}\min\{b_k(\gamma_k+1)\})^{\shortminus 1}$.
\end{lemma}
\begin{proof}
Let $\alpha_\kappa = \bmean - \kappa$ et $t = \frac{1}{\sqrt{\kappa}}$. Then,
$$
\begin{aligned}
\varphi_{\alpha_\kappa}(t) &= \ln(\taumin) + N\sqrt{\kappa} + \sum_{k=1}^N \ln\left(1+\gamma_k^{\shortminus 1} e^{-\frac{1}{\sqrt{\kappa}}b_k(\gamma_k+1)}\right)\\
&\leq \ln(\taumin) + N\sqrt{\kappa} + \sum_{k=1}^N \gamma_k^{\shortminus 1} e^{-\frac{1}{\sqrt{\kappa}}b_k(\gamma_k+1)}\\
&\leq \ln(\taumin) + N\left[\sqrt{\kappa} + (\min\{\gamma_k\})^{\shortminus 1} e^{-\frac{1}{\sqrt{\kappa}}\min\{b_k(\gamma_k+1)\}}\right]\\
&\leq \ln(\taumin) + N\left[\sqrt{\kappa} + \frac{(\min\{\gamma_k\})^{\shortminus 1}\sqrt{\kappa}}{\sqrt{\kappa}+\min\{b_k(\gamma_k+1)\}}\right]\leq \ln(\taumin) + N\Gamma\sqrt{\kappa}\enspace.
\end{aligned}
$$
The last inequality is obtained using $e^{-x}\leq (1+x)^{\shortminus 1}$ for $x>0$.
Therefore, for $\kappa^2 \leq \frac{1}{N\Gamma}\ln(\tau/\tau^-)$, $\varphi_{\alpha_\kappa}(t)\leq \ln(\tau)$ and as a consequence $\varphi^*_{\alpha_\kappa} \leq \ln(\tau) = \varphi^*_{\alpha_\tau}$.
As $\alpha\mapsto\varphi^*_\alpha$ is decreasing, we obtain that $\alpha_\tau \leq \alpha_\kappa$.
\end{proof}
Note that, under its apparent simplicity, this property does not hold for other bound such that Hoeffding or Bennett.

\paragraph{Double bisection search algorithm.}
We now present a fast algorithm to compute $\alpha_\tau$ introduced in~\Cref{prop::main_result}. 
\Cref{algo::refined_bound} consists of two nested bisection searches. The inner one is dedicated to find the minimum in $t$ --  see~\Cref{algo::refined_prob_bound} -- to an arbitrary precision $\epsilon_t$ and, as a consequence, to compute $\varphi^*$ to a precision $M\epsilon_t^2$. This estimation of $\varphi^*$ constitutes the oracle for the outer bisection search. Therefore, the test is more elaborated as it checks whether the decision is sure or not: we only reduce the space by half when the oracle returns value far enough from the target $\ln(\tau)$ i.e., with a distance greater than $M\epsilon_t^2$. If not, then, it means that we obtain at a certain iteration an estimation close enough to $\ln(\tau)$, so we stop at this point. This outer bisection search is a particular case of Probabilistic Bisection Algorithm~(PBA)~\cite{Horstein_1963,Waeber_2013}, where the error term is not necessarily of zero mean but takes values in a small bounded interval.
\begin{algorithm}[!ht]
\caption{Double Bisection Search for confidence bound's computation}\label{algo::refined_bound}
\begin{algorithmic}
\Require $\tau$, $N$, $b_k$, $\sigma_k$, $\epsilon_t$, $\epsilon_\alpha$
\State $\alpha^-,\;\alpha^+\gets 0,\;\bmean - \sqrt{\ln(\tau/\tau^-)\,/\,(N\Gamma)}$ \Comment{Init $\alpha$-bisection}
\State $tol\gets \texttt{false}$
\While{$\alpha^+ - \alpha^- > \epsilon_\alpha$ or $tol = \texttt{false}$} 
  \State $\hat{\alpha} \gets \frac{1}{2}(\alpha^- + \alpha^+)$
  \State $t^-,t^+\gets 0,\shortminus\frac{1}{N(\bmean - \hat{\alpha})}\ln(\tau^-)$ \Comment{Init $t$-bisection}
  \While{$t^+ - t^- > \epsilon_t$ }
     \State $\hat{t} \gets \frac{1}{2}(t^- + t^+)$
     \State $g \gets \frac{d}{dt} \varphi_\alpha(\hat{t})$
      \IfThenElse{$g\geq 0$}
        {$t^+ \gets \hat{t}$}
        {$t^- \gets \hat{t}$}
  \EndWhile
  \State $\hat{\varphi} \gets \varphi_{\hat{\alpha}}(\hat{t})$
  \If{$\hat{\varphi} > \ln(\tau) + M\epsilon_t^2$}
    $\alpha^- \gets \hat{\alpha}$
  \ElsIf{$\hat{\varphi} < \ln(\tau) - M\epsilon_t^2$}
    $\alpha^+ \gets \hat{\alpha}$
  \Else
     $\;tol \gets \texttt{true}$
  \EndIf
\EndWhile
~\\\Return $\hat{\alpha}$
\end{algorithmic}
\end{algorithm}

\paragraph{Termination guarantees.}
The following proposition proves that this algorithm is fast (log convergence) and provides a solution arbitrary close to the optimal solution.
\begin{theorem}\label{prop::convergence_algo}
Let $\tau \in ]\taumin,1]$.
\Cref{algo::refined_bound} ends with a value $\hat{\alpha}$ such that
\begin{equation}
\left|\hat{\alpha} - \alpha_\tau\right| \leq 
\epsilon_\alpha
\,\wedge\,
\sqrt{\frac{2M}{N\min\{m_k\}}}\epsilon_t
\,\wedge\,
\frac{2M}{N\min\{b_km_k\}} \epsilon_t^2
\enspace ,
\end{equation}
where $m_k := \frac{\ln(2+\gamma_k^{\shortminus 1})}{b_k^2 (1+\gamma_k)}$.
Moreover, the total number of iterations $I_\tau$ is bounded:
$$I_\tau\leq
\left\lceil\log_2\left(\frac{\bmean}{\epsilon_\alpha}\right)\right\rceil
\left\lceil\log_2\left(\frac{\sqrt{\Gamma}\ln(1/\tau^-)}
{\epsilon_t \sqrt{N\ln(\tau/\tau^-)}}\right)\right\rceil
\enspace.$$

\end{theorem}
\begin{proof}
At the end of the algorithm, one obtain from the inner bisection that $t^-\leq t^*_{\hat{\alpha}},\hat{t}\leq t^+$ and $|t^+ - t^-|\leq \epsilon_t$. Suppose that the algorithm ends with a value $\hat{\alpha}$ and $\hat{\varphi} = \varphi_{\hat{\alpha}}(\hat{t})$. 
Then, from the mean-value theorem, there exists $t\in[t^-,t^+]$ such that 
$$\left| \frac{d}{dt}\varphi_\alpha(t^-) - \frac{d}{dt}\varphi_\alpha(t^+) \right| = \left|\frac{d^2}{dt^2}\varphi_\alpha(t)\right| (t^+ - t^-) \leq M \epsilon_t\enspace.$$
As the derivative of $\varphi_\alpha$ is decreasing and positive (resp. negative) in $t^-$ (resp. $t^+$), $|\frac{d}{dt}\varphi_\alpha(t)| \leq M\epsilon_t$ for all $t\in[t^-,t^+]$, and using once again the mean-value theorem,
$$
\left|\hat{\varphi} - \varphi^*_{\hat{\alpha}}\right| \leq M\epsilon_t^2 \enspace.
$$
\underline{$1^{\textnormal{st}}$ case}: the algorithm ends with $tol=\texttt{true}$.\\
Therefore (criteria), $\left|\hat{\varphi} - \varphi^*_{\alpha_\tau}\right| \leq M\epsilon_t^2$, and so
$
\left|\varphi^*_{\hat{\alpha}} - \varphi^*_{\alpha_\tau}\right| \leq 2M\epsilon_t^2 \enspace.
$
Using~\Cref{lemma::prop_on_varphi_star}, item (iii), we obtain 
$$
\left|\hat{\alpha} - \alpha_\tau\right| \leq \frac{2 M \epsilon_t^2}
{N\min\{t^*_{\alpha_{\tau}},t^*_{\hat{\alpha}}\}}
\enspace.
$$
Then, using~\Cref{lemma::prop_on_varphi_star}, item (ii), for all $\alpha$,
$$
t^*_\alpha \geq \min_{k|b_k>\alpha} \left\{\frac{1}{b_k(1+\gamma_k)} \ln\left(\frac{\alpha + b_k \gamma_k}{\gamma_k (b_k - \alpha)}\right)\right\} \enspace.
$$
By concavity, $\ln(1+x) \geq \ln(1+z)\min\{x/z,1\}$ for all $x,z \geq 0$. Therefore, for all $k$ such that $b_k>\alpha$ (it exists otherwise $\alpha> \overline{b}$), we obtain:
$$\ln\left(\frac{\alpha + b_k \gamma_k}{\gamma_k (b_k - \alpha)}\right) = \ln\left(1+\frac{\alpha (1+\gamma_k)}{\gamma_k(b_k-\alpha)}\right)\geq \ln\left(2+\frac{1}{\gamma_k}\right)\min\{\frac{\alpha}{b_k-\alpha},1\} \geq \frac{1}{b_k}\ln\left(2+\frac{1}{\gamma_k}\right)\min\{\alpha,b_k\} 
\enspace.$$
Then,
$
t^*_\alpha \geq \alpha \min_{k}\{m_k\}\,\vee\,\min_k\{b_k m_k\}
$.
Therefore, as $\hat{\alpha}$ and $\alpha_{\tau}$ are positive quantities,
$$
\begin{aligned}
t^*_{\alpha_{\tau}}\,\vee\,t^*_{\hat{\alpha}} &\geq \alpha_{\tau} \min_{k}\{m_k\}\,\vee\, \hat{\alpha} \min_{k}\{m_k\}\,\vee\,\min_k\{b_k m_k\}\\
&\geq \left|\hat{\alpha} - \alpha_{\tau}\right|\min_{k}\{m_k\}\,\vee\,\min_k\{b_k m_k\} 
\enspace.
\end{aligned}
$$
Finally, $$
\left|\hat{\alpha} - \alpha_\tau\right| \leq 
\max\left\{
\sqrt{\frac{2M}{N\min\{m_k\}}}\epsilon_t,
\frac{2M}{N\min\{b_km_k\}} \epsilon_t^2
\right\}\enspace.
$$
\underline{$2^{\textnormal{nd}}$ case}: the algorithm ends with $\left|\alpha^+ - \alpha^-\right|\leq \epsilon_\alpha$.\\
Then, as $tol=\texttt{false}$, at each iteration, $\left|\hat{\varphi} - \ln(\tau)\right| \geq M\epsilon_t^2$, and so $\alpha_\tau$ lies in $[\alpha^-,\alpha^+]$.
Therefore, $\left|\hat{\alpha} - \alpha_\tau\right| \leq \left|\alpha^+ - \alpha^-\right|\leq \epsilon_\alpha$.

~\\Besides, denoting by $I_t$ the number of iterations for the inner bisection search,
we have 
$$I_t \leq \left\lceil\log_2\left(\frac{-\ln(\tau^-)}{N(\bmean - \alpha)\epsilon_t}\right)\right\rceil\enspace.$$
Then, as $\bmean - \alpha \geq \sqrt{\frac{1}{N\Gamma}\ln(\tau/\tau^-)}$ (see~\Cref{prop::alpha_lower_b}), 
$I_t \leq \left\lceil\log_2\left(\frac{\sqrt{\Gamma}\ln(1/\tau^-)}{\sqrt{N\ln(\tau/\tau^-))}\epsilon_t}\right)\right\rceil$.
Furthermore, denoting by $I_\alpha$ the number of iterations for the outer bisection search,
we have
$$I_\alpha \leq \left\lceil\log_2\left(\frac{\bmean}{\epsilon_\alpha}\right)\right\rceil\enspace.$$
\end{proof}

\Cref{prop::convergence_algo} proves that~\Cref{algo::refined_bound} is fast (log convergence), and provides a solution with an arbitrary precision. Note that the number of iterations is impacted by the distance of $\tau$ from the minimal value $\tau^-$. In fact, very close to $\tau^-$, the minimizer $t^*_\alpha$ tends to $+\infty$, and therefore the width of the bisection search space becomes large. Nonetheless, for reasonable error tolerance $\tau$, the algorithm takes very few iterations. Meanwhile, the precision of $\hat{\alpha}$ does not depend on $\tau$.

\paragraph{Numerical experiments.}

We use the instances developed in~\Cref{tab::random_var_comp}. 
The results are depicted in~\Cref{fig::confidence_results} for 1000 realizations, and are fast to obtain (few seconds in total). Of course, the  confidence bounds we obtain are greater than the value computed with normal distributions, and so all values are greater than 1. We recover the superiority of the studied bound compared to the standard Bennett's inequality. Besides, Chebyshev-Cantelli's bound is only valuable for low probability level, and becomes inefficient for probabilities close to 1.

\begin{figure}[!ht]
  \centering
  \includegraphics[width=0.75\linewidth,clip=true, trim = 1.cm 0.1cm 1.5cm 1.5cm]{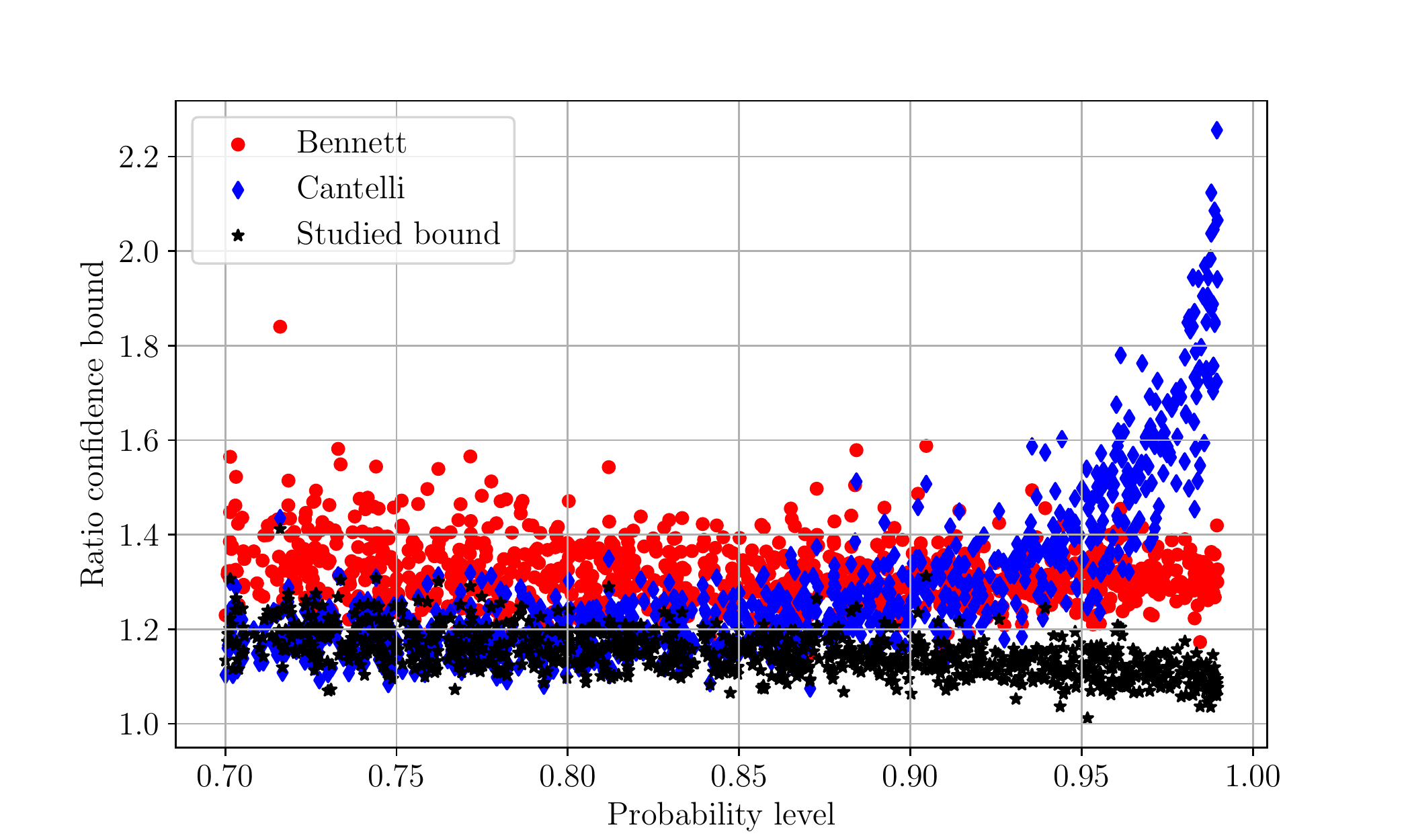}
  \caption{Random instances, made of $N=10$ heterogeneous variables. \\For different probability levels ($1-\tau$) and different inequalities, we display the value of the confidence bound normalized by the normal case, i.e., the (exact) value for normal distributions. }
  \label{fig::confidence_results}
\end{figure}

\section{Application to Chance-Constrained Programming}\label{sec::application}

In the two next subsections, we will use the proposed concentration inequality~\eqref{eq::pos_sum}-\eqref{eq::qq_sum} into CCP problems of the form~\eqref{eq::CCP} with \emph{individual} bilinear constraints of the form $g(x,\xi) = \xi^T x$ - $b$, with $b\in\bbR$. The first application (knapsack problems) is of a combinatorial nature and contains only one chance constraint whereas the second one (Support Vector Machine problems) is a continuous problem but contains as many chance constraints as training points.

\subsection{Chance-constrained binary Knapsack problem}\label{sec::knapsack}
Let us consider the Knapsack problem with random weights, stated as a chance-constrained problem:

\begin{equation}\label{eq::CKP}\tag{CKP}
\begin{aligned}
\max_{y \in\{0,1\}^N} &\;\pi^T y\\
\text{s.t}\quad&\; \Prob\left[\omega^T y \geq C\right]\leq \tau
\end{aligned}
\end{equation}
where $\pi\in\bbR^N$ denotes the utility of each item, $\omega\in \bbR^N$ denotes the random weights of each item, and $C\in \bbR$ is the maximum budget.

If $\omega_k$ follows a normal distribution $\mathcal{N}(\overline{\omega}_k,\sigma_k)$, the chance constraint can be computed \emph{exactly} by the following problem, see e.g.~\cite[Theorem 10.4.1]{Prekopa_1995} :
\begin{equation}\tag{$\text{CKP-N}$}
\begin{aligned}
\max_{y \in\{0,1\}^N} &\;\pi^T y\\
\text{s.t}\quad&\;  \Phi^{\shortminus 1}(1-\tau)\sqrt{ y^T \Sigma y} + \overline{\omega}^T y\leq C
\end{aligned}
\label{eq::CKP-N}
\end{equation}
where $\Phi$ is the cumulative distribution for the standard normal distribution.
In this specific setting, Han et al.~\cite{Han_2015} provides efficient algorithms to obtain robust solutions.

Here, we focus on random weights whose distributions are unknown, but where the two first moments are available, as well as upper bounds.
This distributionally robust approach has been firstly studied by Calafiore and El Ghaoui~\cite{Calafiore_2006}, where they focused on Hoeffding-type approximations. Recently, Ryu and Park~\cite{Ryu_2021} proposed to repeatedly solve ordinary binary knapsack subproblems to deduce bounds on SOCP approximations (such as Chebyshev-Cantelli). As we will compare the different approximations in the sequel, we first recall the next two classical results: let us define $B = \diag(b^2)$ and $\Sigma = \diag(\sigma^2)$, then
\begin{enumerate}[label=(\roman*)]
\item (Hoeffding) the problem~\eqref{eq::CKP-H} is a valid conservative approximation of~\eqref{eq::CKP}
\begin{equation}\tag{$\text{CKP-H}$}
\begin{aligned}
\max_{y \in\{0,1\}^N} &\;\pi^T y\\
\text{s.t}\quad&\;  \sqrt{2\ln(1/\tau)}\sqrt{ y^T B y} + \overline{\omega}^T y\leq C
\end{aligned}
\label{eq::CKP-H}
\end{equation}

\item (Chebyshev-Cantelli) the problem~\eqref{eq::CKP-C} is a valid conservative approximation of~\eqref{eq::CKP}
\begin{equation}\tag{$\text{CKP-C}$}
\begin{aligned}
\max_{y \in\{0,1\}^N} &\;\pi^T y\\
\text{s.t}\quad&\;  \sqrt{\tfrac{1}{\tau} - 1}\sqrt{ y^T \Sigma y} + \overline{\omega}^T y\leq C
\end{aligned}
\label{eq::CKP-C}
\end{equation}
\end{enumerate}
This comparison is inspired by the work of Peng, Maggioni and Lisser~\cite{Peng_2022} where first-order bounds (Hoeffding and an approximation of Bernstein bound) are used in the continuous knapsack problem, and compared to exact SOCP relaxation for normal variables. 

Using~\cref{prop::mgf} with $X_k := \omega_k$ and $d:= C - \overline{\omega}^T y$, we obtain (tighter) conservative approximation. A lower bound of~\eqref{eq::CKP} is then obtained by solving:
\begin{equation}\tag{$\overline{\text{CKP}}$}
\begin{aligned}
\max_{\substack{y \in\{0,1\}^N\\t\geq 0}} &\;\pi^T y\\
\text{s.t}\quad&\;  t\left[\overline{\omega}^Ty-C\right]
+ \sum_{k=1}^N \ln\left(\frac{\gamma_k e^{y_k b_k}+ e^{-t y_k b_k \gamma_k}}{1+\gamma_k}\right)\leq \ln(\tau)
\end{aligned}
\label{eq::relax_CKP}
\end{equation}
The constraint contains bilinear terms $ty_k$. A naïve approach could be to consider a Fortet linearization of the bilinear terms, see e.g.~\cite{Fortet_1960}. Here, using the structure of the constraint, we succeed in reformulating the constraint. Let us consider the change of variable $z := 1/t$ and divide the constraint by $t$:
\begin{equation}
\eqref{eq::relax_CKP} \iff \left\{
\begin{aligned}
\max_{\substack{y \in\{0,1\}^N\\z\geq 0}} &\;\pi^T y\\
\text{s.t}\quad&\;  \overline{\omega}^Ty
+ \sum_{k=1}^N z\ln\left(\frac{\gamma_k e^{\frac{y_k}{z} b_k}+ e^{-\frac{y_k}{z} b_k\gamma_k}}{1+\gamma_k}\right)\leq C + z\ln(\tau)
\end{aligned}\right.
\end{equation}
\begin{prop}\label{prop::convex_knapsack}
For every $\gamma,b\geq 0$, the function $\Psi^+_{\gamma,b}:\bbR\times\bbR_+\to \bbR_+$, defined as,
$$\Psi^+_{\gamma,b}(y,z) := z\ln\left(\frac{\gamma e^{\frac{y}{z} b}+ e^{-\frac{y}{z} b\gamma}}{1+\gamma}\right)\enspace,$$ is jointly convex. Therefore, the problem~\eqref{eq::relax_CKP} -- without the integrity condition -- is convex.
\end{prop}
\begin{proof}
From elementary calculation, the Jacobian and Hessian of $\Psi^+$ are respectively
$$
\begin{aligned}
J_{ \Psi^+_{\gamma,b}}(y,z) &= \begin{bmatrix}
b-\frac{b(1+\gamma)}{1+\gamma e^{\frac{y}{z} b (1+\gamma)}} &
\ln\left(\frac{\gamma + e^{-\frac{y}{z} b (1+\gamma)}}{1+\gamma}\right) + \frac{b(1+\gamma)y}{z\left(1+\gamma e^{\frac{y}{z} b (1+\gamma)}\right)}
\end{bmatrix}\\
H_{ \Psi^+_{\gamma,b}}(y,z) &= \gamma e^{\frac{y}{z} b (1+\gamma)}\left(\frac{b (1+\gamma)}{1+\gamma e^{\frac{y}{z} b (1+\gamma)}}\right)^2\begin{bmatrix}
\frac{1}{z} & -\frac{y}{z^2}\\
-\frac{y}{z^2} & \frac{y^2}{z^3} 
\end{bmatrix}
\end{aligned}
$$
As $\text{Tr}\left(H_{\Psi_{\gamma,b}}(y,z)\right)\geq 0$ and $\text{det}\left(H_{\Psi_{\gamma,b}}(y,z)\right)= 0$, the Hessian is always a positive semi-definite matrix, and the function is jointly convex. 
\end{proof}
\begin{remark}
We can directly obtain the convexity of the function by noting that $\Psi^+_{\gamma,b}$ is the composition of the perspective function~\cite{Combettes_2017} with a particular log-sum-exp function. This proposition is generalized in~\cite{Nemirovski_2007} to a class of moment-generating function's estimators. Nonetheless, we make explicit the Jacobian and the Hessian of the function, since it will be necessary for numerical optimization.
\end{remark}
By~\Cref{prop::convex_knapsack}, Problem~\eqref{eq::CKP} reduces to a convex Mixed-Integer Non-Linear Programming (MINLP) problem. In fact, there is a unique nonlinear constraint (the budget constraint). For such problems with a moderate degree of nonlinearity, cutting-plane methods~\cite{Westerlund_1995} are known to be efficient. In this approach, the nonlinear constraint is described by a set of linear constraints, incrementally built by adding at each iteration a new cutting-plane of the original constraint. The subproblem at iteration $j$ is then expressed as: 
\begin{equation}\label{eq::knapsack_cutting_planes}
\left(y^{(j+1)},z^{(j+1)}\right)=\argmax_{\substack{y \in\{0,1\}^N\\z\geq 0}} \left\{\pi^T y\,\left\vert\,
\overline{\omega}^Ty
+ \sum_{k=1}^N \left<s^{(i)}_k,\begin{pmatrix}y_k\\z\end{pmatrix}\right>\leq C + z\ln(\tau),\,1\leq i\leq j
\right.\right\}\enspace,
\end{equation}
where $s^{(i)}_k := \nabla \Psi^+_{\gamma_k,b_k}\left(y^{(i)}_k,z^{(i)}\right)$.
The convergence of this approach has been proved, see e.g.~\cite[Theorem 9.6]{Bonnans_2006}, and is achieved in a finite number of steps for this particular problem (there is a finite number of knapsack-filling scenarios). Note that the cuts are dynamically added in the branch-and-bound at each node where an integer solution is found (\emph{Lazy constraint}), so that the solver does not need to perform a complete MILP solving at each iteration.

As an alternative to the model~\eqref{eq::relax_CKP}, we also introduce a convex reformulation of the problem under Bernstein's inequality: let suppose that $|\omega_k - \overline{\omega}_k|\leq b_k$ (and not only $\omega_k - \overline{\omega}_k\leq b_k$), then a valid conservative estimation of $\eqref{eq::CKP}$ can be obtained by replacing the probabilistic constraint by Bernstein's inequality (see e.g.~\cite{Boucheron_2004} for more details on this inequality). We obtain the following formulation: 
\begin{equation}\tag{$\text{CKP-B}$}
\begin{aligned}
\max_{y \in\{0,1\}^N} &\;\pi^T y\\
\text{s.t}\quad&\;   \exp \left(-{\frac {{\tfrac {1}{2}}(C-\overline{\omega}^Ty)^2}{\sum _{i=1}^{N}y_k^2\sigma_k^2+{\tfrac {1}{3}}z (C-\overline{\omega}^Ty)}}\right)\leq \tau\\
&\; z \geq b_k y_k,\;1\leq k\leq N
\end{aligned}
\label{eq::CKP-B}
\end{equation}
\begin{prop} \label{prop::CKP-B}
Problem~\eqref{eq::CKP-B} is equivalent with the following problem:
\begin{equation}
\begin{aligned}
\max_{y \in\{0,1\}^N,z\geq 0} &\;\pi^T y\\
\text{s.t}\quad&\;  {\tfrac {1}{3}}\ln(1/\tau)z + \sqrt{\left(y^T\;z\right)\Gamma\begin{pmatrix}
 y \\z
 \end{pmatrix}} +\overline{\omega}^Ty\leq C\\
 &\; z \geq b_k y_k, \; 1\leq k \leq N
\end{aligned}
\end{equation}
where $\Gamma = \text{diag}\begin{bmatrix}
 (2\ln(1/\tau)\sigma^2_k)_{1\leq k\leq N} \\ {\tfrac {1}{9}}\ln(1/\tau)^2
 \end{bmatrix}$. Therefore, problem~\eqref{eq::CKP-B} -- without the integrity condition -- is convex.
\end{prop}
\begin{proof}
We reformulate the constraint so that we end up with a convex reformulation:
$$
\begin{aligned}
 &\exp \left(-{\frac {{\tfrac {1}{2}}t^2}{\sum _{i=1}^{N}y_k^2\sigma_k^2+{\tfrac {1}{3}}z t}}\right)\leq \tau,\quad t=C-\overline{\omega}^Ty, \;z = \max_{k}\{b_k y_k\}\\
 \iff & \ln(1/\tau)\left[\sum _{i=1}^{N}y_k^2\sigma_k^2+{\tfrac {1}{3}}z t\right] \leq \frac{1}{2}t^2\\
 \iff & \ln(1/\tau)\sum _{i=1}^{N}y_k^2\sigma_k^2 \leq \frac{1}{2}\left[t - {\tfrac {1}{3}}\ln(1/\tau)z\right]^2 - {\tfrac {1}{18}}\ln(1/\tau)^2z^2\\
 \iff & \sqrt{2\ln(1/\tau)\sum _{i=1}^{n}y_k^2\sigma_k^2 + {\tfrac {1}{9}}\ln(1/\tau)^2z^2} \leq C-\overline{\omega}^Ty - {\tfrac {1}{3}}\ln(1/\tau)z\\
 \iff & \sqrt{\begin{pmatrix} y^T & z\end{pmatrix} \Gamma\begin{pmatrix}
 y \\z
 \end{pmatrix}} \leq C-\overline{\omega}^Ty - {\tfrac {1}{3}}\ln(1/\tau)z
\end{aligned}
$$
Note that, in the optimization problem, it is sufficient to consider $z\geq \max_k\{b_k y_k\}$ as the optimization will search for the lowest value possible ($z$ only appears on the constraint above), ans so the constraint will be naturally saturated.
\end{proof}


The SOCP formulation introduced in~\Cref{prop::CKP-B} provides an alternative conservative approximation, which can be directly compared to the classical Chebyshev approximation, as they both belongs to the same class of problem. In contrast, the formulation~\eqref{eq::relax_CKP} is not expressed as a cone programming, but we provide in~\Cref{app::conic_reformulation} a reformulation where constraints are expressed via exponential cones.

\begin{remark}
We already know (proved theoretically above and highlighted by~\Cref{fig::res}) that \eqref{eq::relax_CKP} gives better solution than all other formulations (apart from exact one in the case of normally distributed weights), as the set of admissible solutions is larger.
Note also that we did not provide optimization model for the bound developed in~\cite{Jebara_2018} and~\cite{Bennett_1962}, as a convex expression of the chance constraint is all but immediate to obtain (if it exists).
\end{remark}

\paragraph{Numerical results.}
In order to obtain chance-constrained instances, we adapted deterministic instances from the literature\footnote{The instances are extracted from the website
\url{http://artemisa.unicauca.edu.co/~johnyortega/instances_01_KP/}
We use the set of instances \texttt{knapPI\_\{X\}\_1000\_1} where \texttt{X} goes from \texttt{1\_100} to \texttt{2\_10000} (the second number stands for the number of items in the instances), see~\Cref{table::knapsack_results}.}, see~\cite{Pisinger_2005}, by adding a maximum standard deviation of 5\% of the original weight (taken as mean value), and setting the maximum value to $b = 5\sigma$. Note that for normal distribution, the probability of exceeding $\overline{\omega} + 3\sigma$ is $0.997$. Finally, the maximum probability error $\tau$ is taken to 3\%.

We use \texttt{Cplex v12.10} as a MILP solver and the tests are performed on a laptop \texttt{Intel Core i7 @2.20GHz\,$\times$\,12}. The MIP gap tolerance is taken to 0.001\% and the Integrity tolerance to 1e-8. The tests shows that the cutting-plane method adds very few cuts. For instance, the solver added 190 cuts for the instance \texttt{1\_10000}.

\begin{table}[!ht]
\centering
\small
\begin{tabular}{|l||r|r||r|r|r||r||r|r|}
\hline
Instance & KP & \eqref{eq::CKP-N} & \eqref{eq::relax_CKP} & Prob. & Time & \eqref{eq::CKP-B} & \eqref{eq::CKP-C} & \eqref{eq::CKP-H}\\
\hline\hline
\texttt{1\_100}     & 9147     & 8842   & 8817    & 0.19 & 0.1  & 8719    & 8817    & 8150\\\hline
\texttt{1\_200}     & 11238    & 11227  & 10962   & 0.81 & 0.1  & 10682   & 10832   & 10353\\\hline
\texttt{1\_500}     & 28857    & 28606  & 28405   & 2.11 & 0.4  & 28152   & 28127   & 27924\\\hline
\texttt{1\_1000}    & 54503    & 54105  & 53836   & 1.58 & 0.65 & 53617   & 53267   & 52109\\\hline
\texttt{1\_2000}    & 110625   & 110130 & 109779  & 2.95 & 1.8  & 109621  & 109148  & 107228\\\hline
\texttt{1\_5000}    & 276457   & 275685 & 275220  & 2.99 & 33.4 & 275068  & 274151  & \textit{271160}\\\hline
\texttt{1\_10000}   & 563647   & 562560 & 561968  & 3.00 & 97.4 & 561809  & 560387  & \textit{556126}\\\hline\hline
\texttt{2\_100}     & 1514     & 1513   & 1512    & 0.82 & 0.1  & 1456    & 1476    & 1395\\\hline
\texttt{2\_200}     & 1634     & 1619   & 1594    & 0.69 & 0.2  & 1558    & 1592    & 1508\\\hline
\texttt{2\_500}     & 4566     & 4537   & 4504    & 2.31 & 0.5  & 4472    & 4472    & 4348\\\hline
\texttt{2\_1000}    & 9052     & 9008   & 8970    & 2.87 & 1.52 & 8951    & 8927    & 8761\\\hline
\texttt{2\_2000}    & 18051    & 17991  & 17946   & 2.85 & 4.0  & 17925   & 17872   & \textit{17635}\\\hline
\texttt{2\_5000}    & 44356    & 44262  & 44201   & 2.86 & 32.7 & 44184   & 44073   & \textit{43696}\\\hline
\texttt{2\_10000}   & 90204    & 90071  & 89996   & 2.99 & 84.2 & 89975   & 89807   & \textit{89265}\\\hline
\end{tabular}
\caption{Results for knapsack instances. 
We compare the two new formulations \eqref{eq::relax_CKP} and \eqref{eq::CKP-B} to the existing methods \eqref{eq::CKP-H} and \eqref{eq::CKP-C}. The method KP corresponds to the deterministic case, and \eqref{eq::CKP-N} corresponds to a normally-distributed uncertainty. For \eqref{eq::relax_CKP}, we also provide the probability error and the computational time. When the objective is in italic, the solver does not succeed to prove the optimality in the given time. }
\label{table::knapsack_results}
\end{table}

The numerical tests show the efficiency of the proposed relaxation: the use of the second-order information leads to a substantial improvement of the optimal objective-function value, compared to the classical Hoeffding bound. Besides, this method appears to be easy tractable, as we were able to solve instances of $10000$ items in less than two minutes. Note that the relaxation seems to be a bit more tractable than the Hoeffding bound, as the solver cannot prove the optimality of the solution with the desired precision in less than 10 minutes.

 \begin{remark}We present here the results for mixed-integer problems, but the results and the methodology does not exploit the integrity condition of the variables, and so the results and the methodology are still applicable on the (simpler) continuous problem. In particular, a cutting-plane approach still converges.
 \end{remark}

\subsection{Distributionally Robust Support Vector Machine problem}\label{sec::SVM}
Let us consider a dataset of $M$ points $\{x_i, l_i\}$ where each point $x_i$ is a vector of $\bbR^N$. The points lies into two classes, indexed by labels $l_i\in \{-1,1\}$. 
The chance-constrained formulation of the Support Vector Machine (SVM) problem (with soft margin) is defined as follows:
\begin{equation}\label{eq::SVM}\tag{SVM -- CCP}
\begin{aligned}
\min_{w\in\bbR^N,w_0\in\bbR,\xi\in\bbR_+^M} &\quad \frac{1}{2}\|w\|_2^2 + C\sum_{i=1}^M \xi_i\\
\text{s.t.} &\quad \Prob\left[l_i(w^T x_i + w_0) \leq 1 - \xi_i\right]\leq \tau_i\\
&\quad \xi_i \geq 0,\; i=1,\hdots,M
\end{aligned}
\end{equation}
This robust version has been recently studied, see e.g.~\cite{Wang_2015,Khanjani_2022}. Here, we focus on independent noises for each points as in~\cite{Ben_Tal_2010}.
In contrast with the knapsack problem~\eqref{eq::CKP}, the random training points are multiplied by $w$, which can be either positive or negative. Therefore, we can no longer apply~\eqref{eq::pos_sum} and must use~\eqref{eq::qq_sum} which contains absolute values. Moreover, each training feature (point) defines a (nonlinear) chance constraint.
\begin{prop}\label{prop::SVM_convex}
Let~\eqref{eq::relax_SVM} be defined as
\begin{equation}\label{eq::relax_SVM}\tag{$\overline{\text{SVM}}$}
\begin{aligned}
\min_{w\in\bbR^N,w_0\in\bbR,\xi\in\bbR_+^M} &\quad \frac{1}{2}\|w\|_2^2 + C\sum_{i=1}^M \xi_i\\
\text{s.t}\quad&\;  -l_i(w_0 + w^T \overline{x}_i)
+ \sum_{k=1}^N \Psi_{\gamma_{ik},b_{ik}}(w_k,z_i)\leq \xi_i - 1 + z_i\ln(\tau_i),\,1\leq i\leq M
\end{aligned}
\end{equation}
where $\Psi_{\gamma,b}:\bbR\times \bbR_+ \to \bbR_+$ is defined as $\Psi_{\gamma,b}(y,z)=\Psi^+_{\gamma,b}(|y|,z)$.
This problem is a valid conservative approximation of~\eqref{eq::SVM}.
\end{prop}
\begin{proof}
We follow the similar steps as for the knapsack case. In particular, we use~\Cref{prop::mgf} for each chance constraint $i\in\{1,\hdots,M\}$ with $X = x$, $\lambda = -l_i w$ and $d = \xi_i -1 + l-i(w_0+\overline{x})$. Then, we apply the chance of variable $z=1/t$.
\end{proof}
The following proposition shows that the function $\Psi$ keeps the same regularity as $\Psi^+$:
\begin{prop}\label{prop::smooth_C2}
The function $\Psi_{\gamma,b}$ is convex and twice continuously differentiable. Moreover,
\begin{equation}
\Psi_{\gamma,b}(y,z)=\begin{cases}
\Psi^+_{\gamma,b}(y,z),\quad y\geq 0\\
\Psi^+_{\frac{1}{\gamma},b\gamma}(y,z),\quad y\leq 0
\end{cases}
\end{equation}
As a consequence, problem \eqref{eq::relax_SVM} is a \emph{convex} conservative approximation of~\eqref{eq::SVM}.
\end{prop}
\begin{proof}
We have the following direct equalities:
$$
\Psi^+_{\gamma,b}(-y,z) 
 = z\ln\left(\frac{e^{-\frac{y}{z} b}+ \gamma^{\shortminus 1}e^{\frac{y}{z} b\gamma}}{1+\gamma^{\shortminus 1}}\right)
 = z\ln\left(\frac{e^{-\frac{y}{z} (b\gamma)\gamma^{\shortminus 1}}+ \gamma^{\shortminus 1}e^{\frac{y}{z} (b\gamma)}}{1+\gamma^{\shortminus 1}}\right) = \Psi^+_{\gamma^{\shortminus 1},b\gamma}(y,z)
 \enspace.
$$
Furthermore, to check the regularity property, it suffices to verify the condition in $y=0$. $\nabla \Psi^+_{\gamma,b}(0,z) = 0$ for all $\gamma$ and $b$, so $\nabla \Psi_{\gamma,b}(0^-,z) = \nabla \Psi_{\gamma,b}(0^+,z) = 0$. Moreover,
$
H_{\Psi^+_{\gamma,b}}(0,z) = \begin{bmatrix}b^2\gamma & 0\\0 & 0\end{bmatrix} = H_{\Psi^+_{\gamma^{-1},b\gamma}}(0,z) \enspace.
$
Therefore, $\Psi_{\gamma,b}$ is twice continuously differentiable in $y=0$.
\end{proof}

\paragraph{Numerical results.}
In the tests, \eqref{eq::relax_SVM} is implemented using the interior-point nonlinear solver \texttt{IPOPT}\footnote{\url{https://coin-or.github.io/Ipopt/}}. The solver always returns the optimal solution as the problem has been proved to be convex, see~\Cref{prop::smooth_C2}. For comparison, we also implement the robust SVM approximation using Chebyshev-Cantelli inequality -- see e.g.~\cite{Wang_2015} -- which can be efficiently solved by any SOCP solver.

\begin{figure}[!ht]
\centering
\begin{subfigure}{0.49\linewidth}
\centering
	\includegraphics[width=\linewidth,clip=true, trim= 3.2cm 9cm 3cm 9cm]{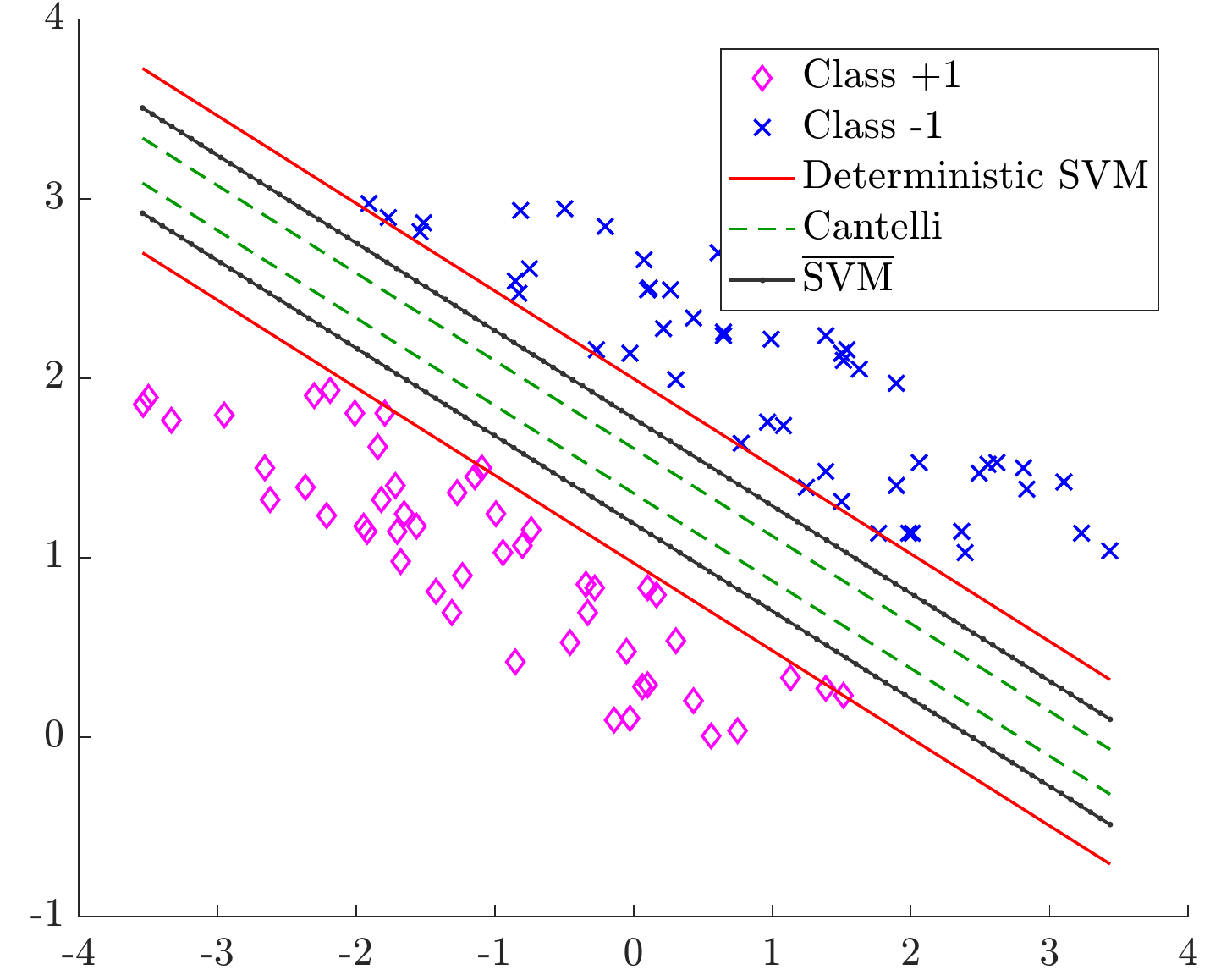}
	\caption{Robust separation \\(no slack variables used)}
\end{subfigure}
\begin{subfigure}{0.49\linewidth}
\centering
	\includegraphics[width=\linewidth,clip=true, trim= 3.2cm 9cm 3cm 9cm]{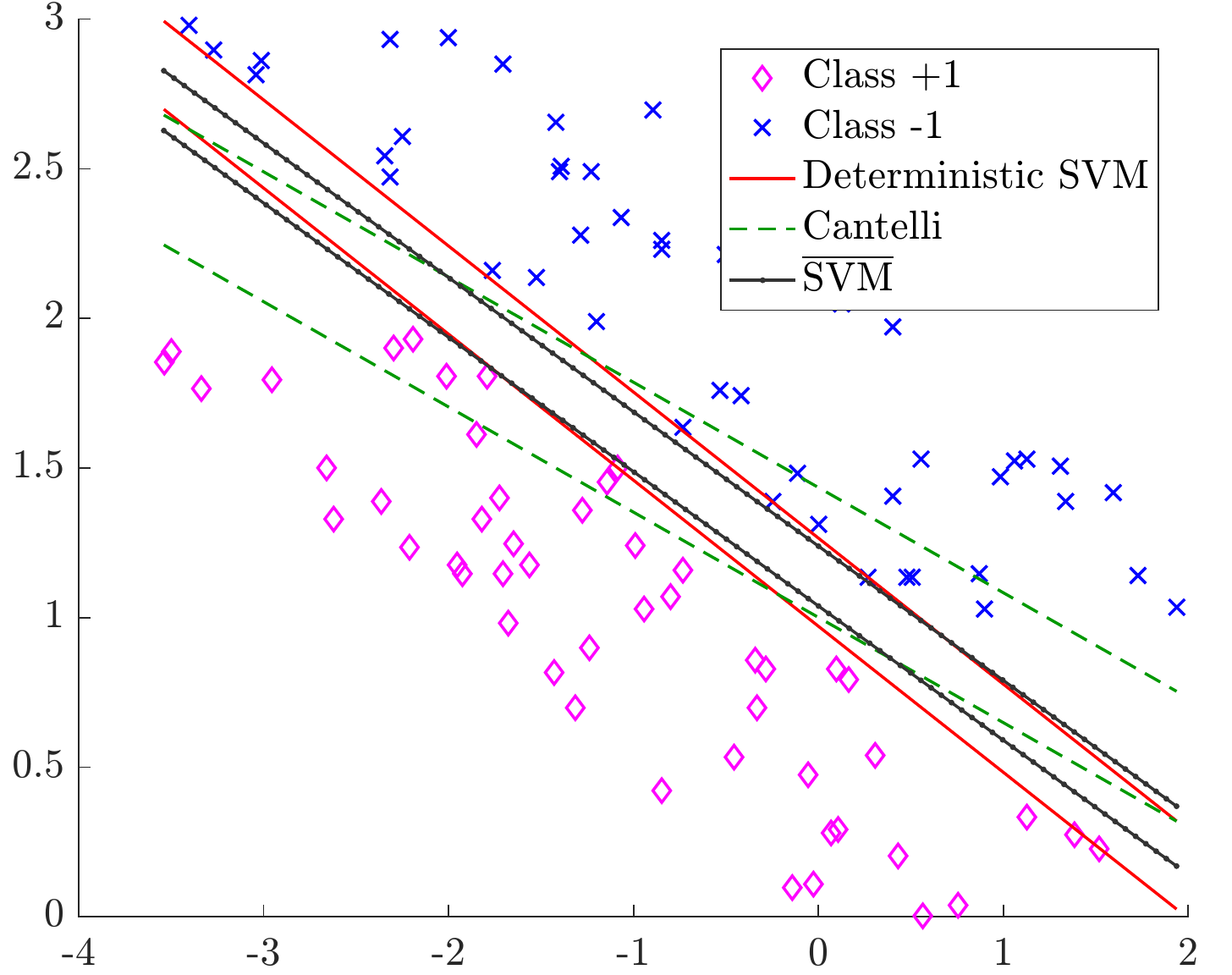}
	\caption{Robust separation\\(active slack variables for some points)}
	\label{fig::svm_tight}
\end{subfigure}
\caption{Two-dimensions SVM with $\tau=0.02$, $M = 100$, $C=100$.\\
We directly represent the margin around the hyperplane (centered between the two lines)}
\label{fig::svm}
\end{figure}

First, we construct 2D instances with linearly separable classes. To compute the standard deviation of each point, we follow the method of~\cite{Wang_2015} by calculating the standard deviation of the training points for each class and then divide by 10. This appears to be reasonable as an uncertainty set for each data point. The results are displayed on~\Cref{fig::svm}.  On the left, the classes are sufficiently distant so that it is not necessary to activate slack variables $\xi_i$. We observe that all the methods find the same hyperplane, but differ on the size of the margin width. As expected, the Chebyshev-Cantelli' inequality is more conservative on this example. On~\Cref{fig::svm_tight}, we reduce the space between the two classes. The points are still linearly separable in the deterministic setting, but are not robustly separable both for Chebyshev and for the proposed method. Nonetheless, we numerically observe that Cantelli relaxation needs to activate more slack variables, and so the optimal value is greater than the one found by the proposed method.

\begin{figure}[!ht]
\centering
\begin{subfigure}{0.45\linewidth}
\centering
	\includegraphics[width=\linewidth,clip=true, trim= 3.5cm 9.5cm 3.8cm 9.5cm]{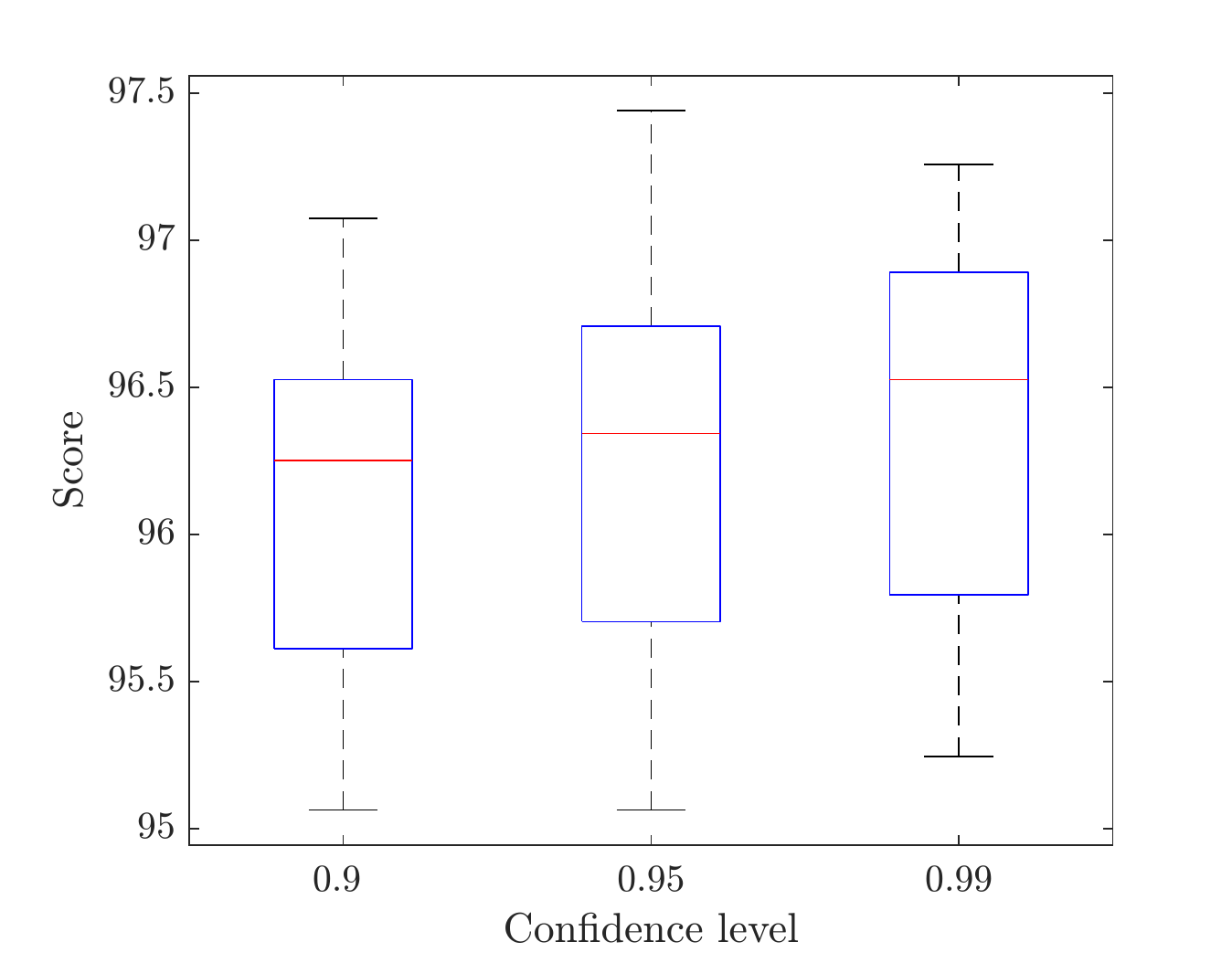}
	\caption{Score with 20\% of training points}
  \label{fig::wis_1}
\end{subfigure}
\begin{subfigure}{0.45\linewidth}
\centering
	\includegraphics[width=\linewidth,clip=true, trim= 3.5cm 9.5cm 3.8cm 9.5cm]{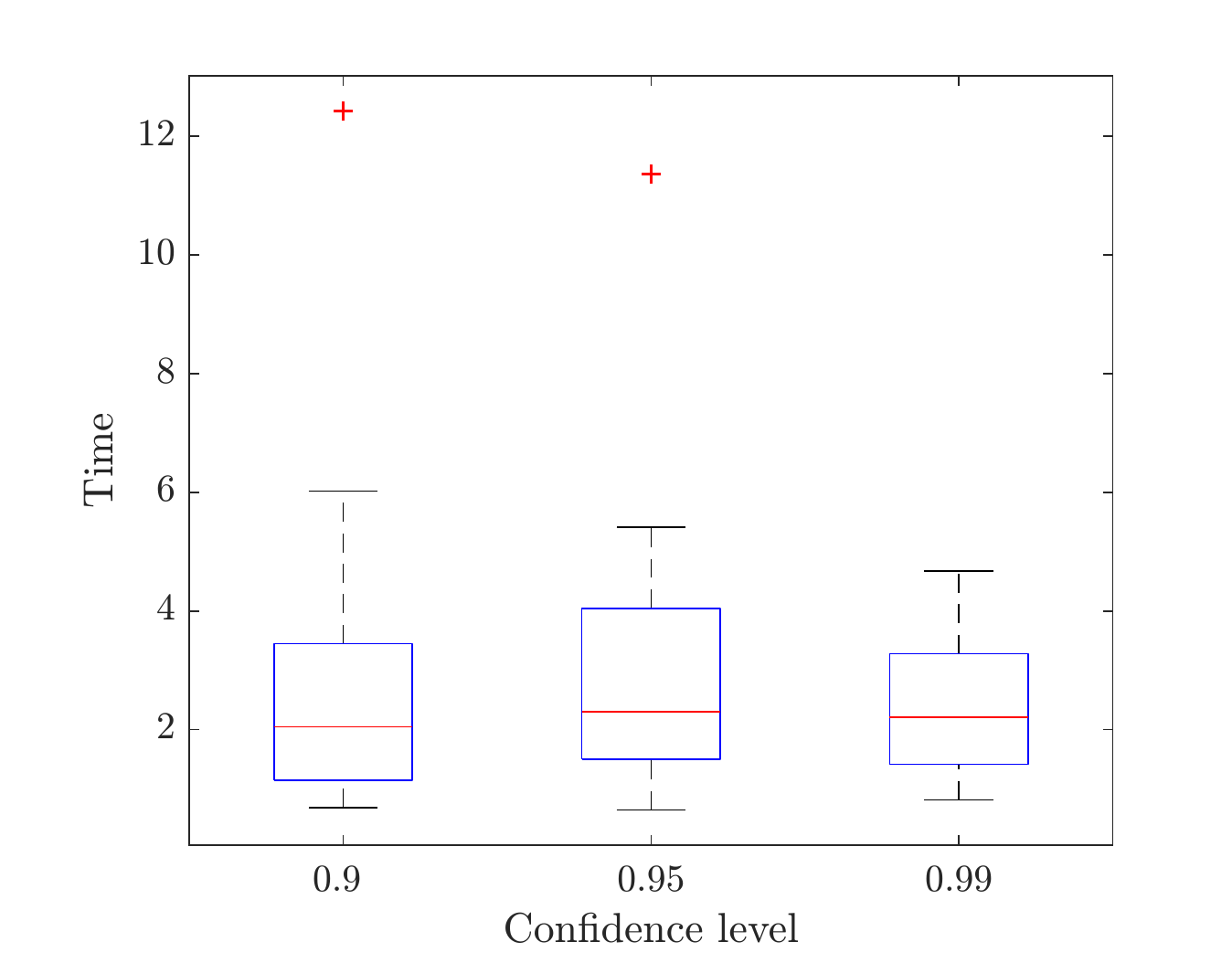}
	\caption{Time (s) with 20\% of training points}
	\label{fig::wis_2}
\end{subfigure}
\begin{subfigure}{0.45\linewidth}
\centering
	\includegraphics[width=\linewidth,clip=true, trim= 3.5cm 9.5cm 3.8cm 9.5cm]{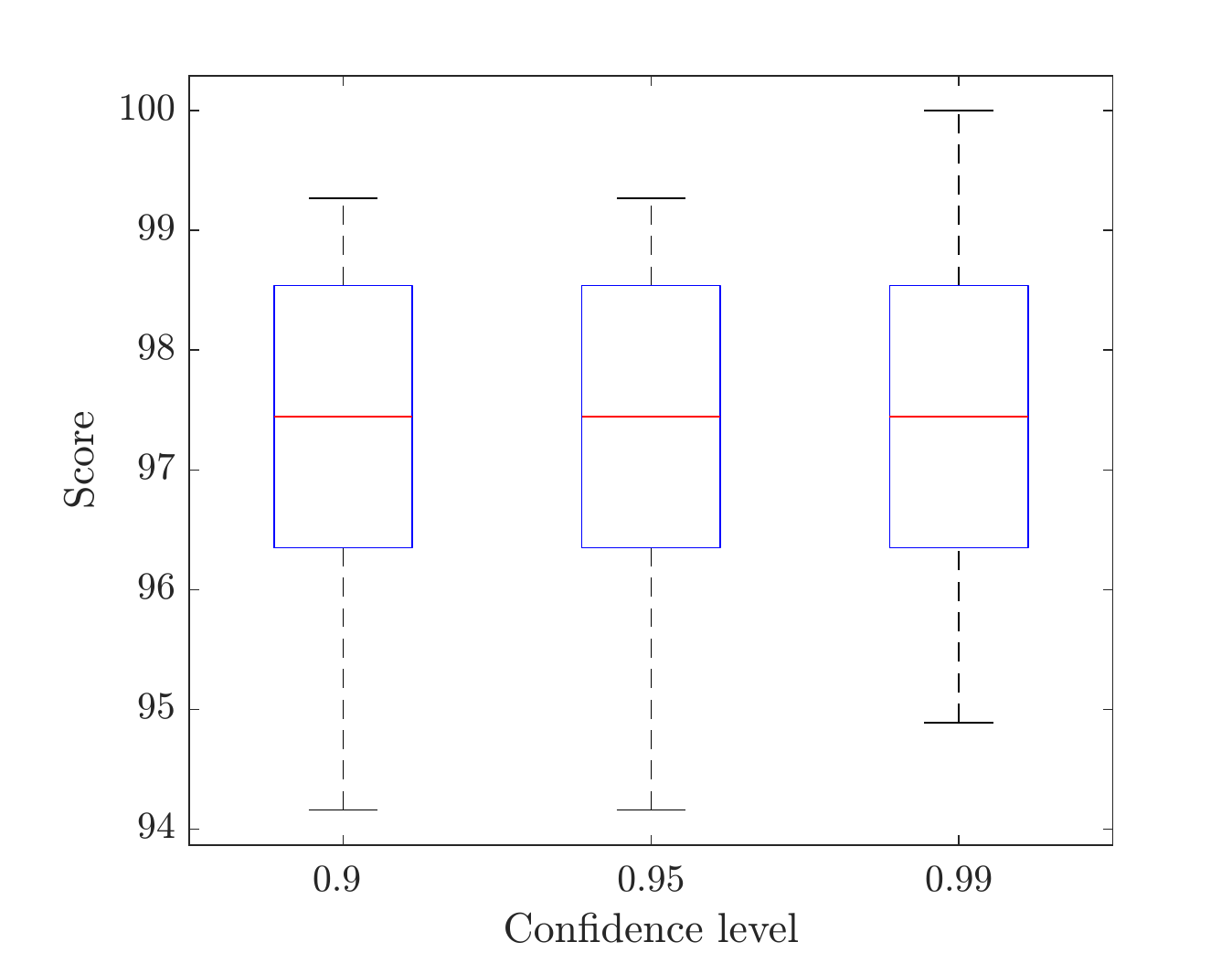}
	\caption{Score with 80\% of training points}
  \label{fig::wis_3}
\end{subfigure}
\begin{subfigure}{0.45\linewidth}
\centering
	\includegraphics[width=\linewidth,clip=true, trim= 3.5cm 9.5cm 3.8cm 9.5cm]{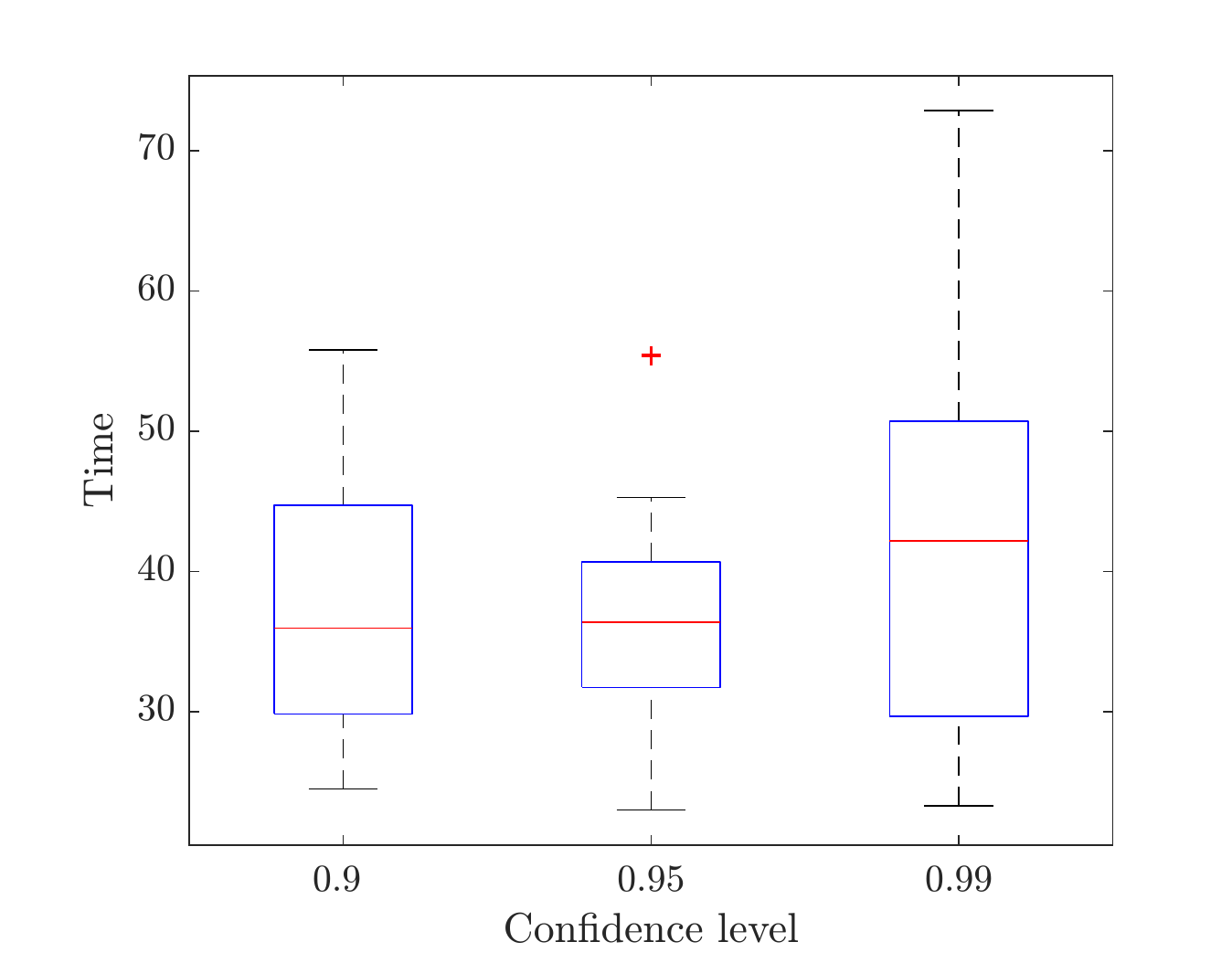}
	\caption{Time (s) with 80\% of training points}
	\label{fig::wis_4}
\end{subfigure}
\caption{Two-dimensions SVM with $\tau=0.02$, $M = 100$, $C=100$}
\label{fig::wis}
\end{figure}

We then use the proposed method on instances from the literature. In particular, we use data on Breast Cancer in the Wisconsin\footnote{\href{https://archive.ics.uci.edu/ml/datasets/Breast+Cancer+Wisconsin+\%28Diagnostic\%29}
{https://archive.ics.uci.edu/ml/datasets/Breast+Cancer+Wisconsin+\%28Diagnostic\%29}}. This data set contains 683 samples of dimension 10, see e.g.~\cite{Wang_2015,Khanjani_2022} for more information on the dataset. \Cref{fig::wis} displays the time and the score (the percentage of test data that satisfies the classification obtained with the training set) for two configurations. We observe that the mean score is always higher than 96\% (same order as in~\cite{Wang_2015,Khanjani_2022}), and the time stays reasonable even for a substantial number of features (more than 500 training points, see~\Cref{fig::wis_2,fig::wis_4}). 

\section{Conclusion and perspectives}
In this paper, we studied a refined Bennett-type inequality, originally developed in the homogeneous setting and extended here to the heterogeneous case. We have shown that this concentration inequality can be used in a wide range of applications. First, we introduce a double bisection search which computes (in logarithmic time) confidence bounds proved to be tighter than other classical approaches. In particular, it outperforms the standard Chebyshev's approach for high probability precision.
Besides, we obtained tight distributionally robust bounds to individual Chance-constrained Programming which can be formulated as convex problem. In particular, we highlighted that the inequality can be inserted into CCP binary knapsack problem while staying tractable (instances of 10\;000 binary variables). Tests on SVM problems have also been performed, obtaining a better separability of the data on instances from the literature (containing up to 500 points). 

Future works will be dedicated to the extension of the results to the independent joint probability constraint case. Moreover, we think that this inequality can be helpful in many concrete applications to estimate more precisely error bounds, especially we will focus on electricity bill estimates.

\printbibliography[heading=bibintoc]

\appendix
\section{Proofs}

\subsection{Comparison of $\mathbb{E}[e^{t X}]$ estimations from the litterature}\label{sec::comparison_moment_generating}
\paragraph{\cite[(c)]{Bennett_1962}$\Rightarrow$\cite[(b)]{Bennett_1962}.}
Suppose that $X - \bbE[X] \leq b$ and $\Var(X)\leq \sigma^2$.
Then, in~\cite{Jebara_2018}, the upper estimator of the moment-generating function is 
$J(t) = 1+\gamma \left(e^{tb}-1-tb\right)$,
where $\gamma = (\sigma / b)^2$.
Besides, in~\cite{Dembo_2010}, the upper estimator is
$$
D(t) := \frac{\gamma e^{t b} 
+ e^{-t\gamma b}} {1+\gamma}\enspace.
$$
If now we consider $D$ as a function of $\gamma$, i.e., $D(t,\gamma) = D(t)$, then, the second partial derivative w.r.t. $\gamma$ is 
$\partial^2_\gamma D(t,\gamma) = \frac{2}{(1+\gamma)^3}\left[e^{_\gamma t} - e^t\right]\leq 0$.
Therefore, $D(t,\cdot)$ is concave for any fixed $t\geq 0$ and 
$$D(t,\gamma)\leq D(t,0) + \gamma \partial_\gamma D(t,0) = J(t)\enspace.$$

\paragraph{\cite[(c)]{Bennett_1962}$\Rightarrow$\cite{Pinter_1989}$\Rightarrow$\cite{Hoeffding_1963}.} 
As $\gamma\mapsto D(t,\gamma)$ is increasing, then $D(t,\gamma)\geq D(t,1)=\cosh(tb)$, which is exactly the bound obtained by Pinter with $a=b$.
As $\cosh(x)\leq \exp(x^2/2)$, we have $D(t,1)\leq e^{(tb)^2/2}$, which is exactly the Hoeffding's estimator.

\paragraph{\cite{From_2013,Zheng_2017}$\Rightarrow$\cite{Hoeffding_1963}.}
Now, until the end of the proof, let us suppose that $X\in[0,1]$, i.e., $a = -\bbE[X]$ and $b = 1-\bbE[X]$.
We denote by $p = \bbE[X]$ the mean value and by $\sigma^2$ the variance.
Then, in~\cite{Hoeffding_1963}, the upper estimator of the moment-generating function is 
$\displaystyle H(t) = e^{tp + t^2/8}\enspace.$
In~\cite{Zheng_2017} and~\cite{From_2013}, the upper estimator of the moment-generating function $
\bbE[e^{t(X-p)}]$ is 
$Z(t):= 1+p(e^t -1)$.
By basic algebra, 
$H'(t)-Z'(t) = \left(p+\frac{t}{4}\right) e^{tp + t^2/8} - pe^t\enspace.$
Then 
$$H'(t)-Z'(t) \geq 0 \iff \ln\left(1+\frac{t}{4p}\right)+t(p-1)+t^2/8\geq 0\enspace.$$
As $\ln(1+x)\geq \frac{x}{1+\tfrac{1}{2}x}$ for $x\geq 0$,
$H'(t) - Z'(t)\geq 0 $ if 
$$\left[\frac{1}{4p}+p-1\right] + t\left[\frac{1}{8}+\frac{p-1}{8p}\right]+t^2\left[\frac{1}{8^2p}\right]\geq 0\enspace.$$
The above condition holds since the discriminant of this second-order equation is zero. Therefore, $H'(t)-Z'(t)\geq 0$, and since $H(0) = Z(0)$, we finally conclude that $H(t)\geq Z(t)$ for $t>0$.

\paragraph{\cite{Cheng_2022}$\Rightarrow$\cite{From_2013,Zheng_2017}.}
In~\cite{Cheng_2022}, the upper estimator is a family of function $C_k$ such that
$$
C_k(t):= 1+k\left(e^{t/k}-1\right)(p - \sigma^2 - p^2) + (\sigma^2+p^2)(e^t - 1)\enspace,
$$
One can prove that $\{C_k(t)\}_k$ is decreasing $\forall t\in \bbR_+$, and
$$
\lim_{k\to\infty} C_k(t) = C_\infty(t) := 1+t(p - q) + q(e^t - 1)\enspace,$$
where $q := \sigma^2+p^2\leq p$. Hence, 
$C_\infty(t) - Z(t) = (p-q)\left(1+t - e^t\right)\leq 0\enspace.$

\paragraph{\cite[(b)]{Bennett_1962}$\Rightarrow$\cite{Cheng_2022}.} 
For $X\in[0,1]$, the upper estimator of~\cite{Jebara_2018} is 
$J(t) = 1+\gamma \left(e^{t(1-p)}-1-t(1-p)\right)$.
Using the notation $\theta = (1-p)^2 \in[0,1]$, we express $C_\infty$ and $J$ in $(\theta,\gamma)$-coordinates:
$$
\begin{cases}
C_\infty(t,\gamma,\theta) = 1+t(1\shortminus\sqrt{\theta})+\left[\gamma \theta + (1\shortminus\sqrt{\theta})^2\right]\left[e^t-1-t\right]\\
J(t,\gamma,\theta) = 1+\gamma \left(e^{t\sqrt{\theta}} - 1 - t\sqrt{\theta}\right)
\end{cases}
$$
Now, the partial derivatives w.r.t $\gamma$ are
$$
\begin{cases}
\partial_\gamma C_\infty(t,\gamma,\theta) = \theta\left[e^t-1-t\right]\\
\partial_\gamma J(t,\gamma,\theta) = e^{t\sqrt{\theta}} - 1 - t\sqrt{\theta}
\end{cases}
$$
The function $[0,1]\ni\theta\mapsto e^{t\sqrt{\theta}} - 1 - t\sqrt{\theta}$ is convex for any $t\geq 0$, and so $\partial_\gamma J\leq\partial_\gamma C_\infty$. As $J(t,0,\theta) = 1 \leq C_\infty(t,0,\theta)$, we conclude that 
$J(t)\leq C_\infty(t)\enspace.$



\section{Conic reformulation}\label{app::conic_reformulation}
First, $\Psi_{\gamma,b}$ can also be written 
$\Psi_{\gamma,b}(y,z) = \max\left\{\Psi^+_{\gamma,b}(y,z), \Psi^+_{\gamma^{\shortminus 1},b\gamma}(y,z)\right\}$. Therefore,
\begin{equation*}
\sum_{k=1}^N \Psi_{\gamma_k,b_k}(y_k,z)\leq u
\iff \begin{cases}
\sum_{k=1}^N v_k\leq u\\
\Psi^+_{\gamma_k,b_k}(y_k,z)\leq v_k\\
\Psi^+_{\gamma^{\shortminus 1}_k,b_k\gamma_k}(y_k,z)\leq v_k
\end{cases}
\end{equation*}
Now, denoting the exponential cone by $\mathcal{K}_{exp} = \{(x,1,x_2,x_3): x_1\geq x_2 e^{x_3/x_2}\}$,
$$
\begin{aligned}
\Psi^+_{\gamma_k,b_k}(y_k,z)\leq v_k
\iff & \gamma_k e^{\frac{y_k}{z} b_k}+ e^{-\frac{y_k}{z} b_k\gamma_k} \leq e^{\frac{v_k}{z}}(1+\gamma_k)\\
\iff & \gamma_k e^{\frac{y_k b_k - v_k}{z} b_k}+ e^{\frac{-y_k b_k\gamma_k - v_k}{z} } \leq 1 +\gamma_k\\
\iff & \begin{cases}
\gamma_k \eta_k + \nu_k \leq (1+\gamma_k)z\\
(\eta_k,z,y_k b_k-v_k)\in\mathcal{K}_{exp}\\
(\nu_k,z,-y_k b_k \gamma_k -v_k)\in\mathcal{K}_{exp}\\
\end{cases}
\end{aligned}
$$

This formulation has a number of variables and (conic) constrains of order $O(NM)$. Therefore, this conic reformulation is only valuable for small to medium instance sizes.

\end{document}